\documentclass[12pt]{article}
\usepackage{graphicx}
\usepackage{amssymb,amsfonts,amsthm}
\usepackage{amsmath,amsthm,amssymb}
\usepackage{amscd}
\usepackage{amsfonts}
\usepackage{mathrsfs}
 
\newcommand{\con}{\operatorname{con}}
\newcommand{\mul}{\operatorname{mul}}
\newcommand{\simp}{\operatorname{sim}}
\newcommand{\var}{\operatorname{var}}

\newtheorem{theorem}{Theorem}[section]
\newtheorem{lemma}[theorem]{Lemma}
\newtheorem{example}[theorem]{Example}
\newtheorem{corollary}[theorem]{Corollary}
\newtheorem{proposition}[theorem]{Proposition}
\newtheorem{fact}[theorem]{Fact}
\newtheorem{remark}[theorem]{Remark}
\newtheorem{sortlemma}{Sorting Lemma}

\newtheorem*{sufcon}{Sufficient Condition}

\theoremstyle{definition}

\numberwithin{equation}{section}

\begin{document}

\title{Classification of limit varieties\\ of $\mathscr J$-trivial monoids}

\author{S. V. Gusev and O. B. Sapir}

\date{}

\maketitle

\renewcommand{\thefootnote}{\fnsymbol{footnote}}
\footnotetext{The first-named author is supported by the Ministry of Science and Higher Education of the Russian Federation (project FEUZ-2020-0016).}    
\renewcommand{\thefootnote}{\arabic{footnote}} 

\begin{abstract} 
A variety of algebras is a \textit{limit variety} if it is non-finitely based but all its proper subvarieties are finitely based.

We present a new pair of limit varieties of monoids and show that together with the five limit varieties of monoids previously discovered by Jackson, Zhang and Luo and the first-named author, there are exactly seven limit varieties of $\mathscr J$-trivial monoids.
\end{abstract}

\section{Introduction}
\label{sec: intr}

A variety of algebras is called \textit{finitely based} (abbreviated to FB) if it has a finite basis of its identities, otherwise, the variety is said to be \textit{non-finitely based} (abbreviated to NFB).
Much attention is paid to studying of FB and NFB varieties of algebras of various types. 
In particular, the FB and NFB varieties of semigroups and monoids have been the subject of intensive research (see the survey~\cite{Volkov-01}).

A variety is \textit{hereditary finitely based} (abbreviated to HFB) if all its subvarieties are FB. 
A variety is called a \textit{limit variety} if it is NFB but every its proper subvariety is FB. 
Limit varieties play an important role because every NFB variety contains some limit subvariety by Zorn's lemma. 
It follows that a variety is HFB if and only if it does not contain any limit variety. 
So, if one manages to classify all limit varieties within some class of varieties, then this classification implies a description of all HFB varieties in this class.

Limit varieties are very rare. 
Only five explicit examples of limit varieties of monoids are known so far.
The first two examples of limit monoid varieties $\mathbf L$ and $\mathbf M$ were discovered by Jackson~\cite{Jackson-05} in 2005 (the formal definitions of these varieties will be given in Subsection~\ref{construction}). 
In 2013, Zhang found a NBF variety of monoids that does not contain the varieties $\mathbf L$ and $\mathbf M$~\cite{Zhang-13} and, therefore, she proved that there exists a limit variety of monoids that differs from $\mathbf L$ and $\mathbf M$. 
In~\cite{Zhang-Luo}, Zhang and Luo pointed out an explicit example of such variety. 

If $S$ is a semigroup, then the monoid obtained by adjoining a new identity element to $S$ is denoted by $S^1$. 
If $M$ is a monoid, then the variety of monoids generated by $M$ is denoted by $\var M$. 
Let $\mathbf A^1=\var A^1$, where
\[
A=\langle a,b,c\mid a^2=a,\,b^2=b,\,ab=ca=0,\,ac=cb=c\rangle=\{a,b,c,ba,bc,0\}.
\]
The semigroup $A$ was introduced by its multiplication table and shown to be FB in \cite[Section~19]{Lee-Zhang}.
Its presentation was recently suggested by Edmond W. H. Lee. 
If $\mathbf V$ is a monoid variety, then  $\overleftarrow{\mathbf V}$ denotes the variety \textit{dual to} $\mathbf V$, i.e., the variety consisting of monoids anti-isomorphic to monoids from $\mathbf V$.
The variety $\mathbf A^1\vee\overleftarrow{\mathbf A^1}$ is the third example of limit variety of monoids~\cite{Zhang-Luo} mentioned in the previous paragraph.
The fourth and the fifth examples of limit monoid varieties $\mathbf J$ and $\overleftarrow{\mathbf J}$ are provided in~\cite{Gusev-SF} (the formal definition of the variety $\mathbf J$ will be given in Subsection~\ref{construction}).

In~\cite{Green-51}, Green introduces five equivalence relations on a semigroup. 
These relations are collectively referred to as Green's relations, and play a fundamental role in studying semigroups. 
We recall the definition of one of them, which we use in present paper. 
For elements $a$ and $b$ of a semigroup $S$, Green's relation $\mathscr J$ is defined by
$$
a \mathscr J b\ \text{ if and only if }\ S^1 aS^1 = S^1 bS^1, \text{ i.e., } a \text{ and } b \text{ generate the same ideal.}
$$
A semigroup $S$ is called $\mathscr J$-\textit{trivial} if Green's relation $\mathscr J$ is the equality relation.

It turns out that all  known examples of limit varieties of monoids are varieties of $\mathscr J$-trivial monoids. 
In this article, we present a new pair of limit varieties of monoids $\mathbf K$ and $\overleftarrow{\mathbf K}$ and provide the following classification of limit varieties of $\mathscr J$-trivial monoids, which is the main result of the article.

\begin{theorem} 
\label{T: main} 
A variety of $\mathscr J$-trivial monoids is HFB if and only if it excludes the varieties $\mathbf A^1\vee\overleftarrow{\mathbf A^1}$, $\mathbf J$, $\overleftarrow{\mathbf J}$, $\mathbf K$, $\overleftarrow{\mathbf K}$, $\mathbf L$ and $\mathbf M$. 
Consequently, there are precisely seven limit varieties of $\mathscr J$-trivial monoids.
\end{theorem}

A monoid is \textit{aperiodic} if all its subgroups are trivial.
We note that several classifications of limit varieties of monoids were obtained earlier in some other classes. 
Namely, Lee proved in~\cite{Lee-09} the uniqueness of the limit varieties $\mathbf L$ and $\mathbf M$ in the class of varieties of finitely generated aperiodic monoids with central idempotents. 
In~\cite{Lee-12}, Lee generalized the result of~\cite{Lee-09} and established that  $\mathbf L$ and $\mathbf M$ are the only limit varieties within the class of varieties of aperiodic monoids with central idempotents. 
Just recently, the first-named author~\cite{Gusev-JAA} proved that a variety of aperiodic monoids with commuting idempotents is HFB if and only if it excludes the varieties ${\mathbf J}$, $\overleftarrow{\mathbf J}$, ${\mathbf L}$ and ${\mathbf M}$.

The article consists of six sections. 
Section~\ref{preliminaries} contains definitions, notation, auxiliary results and introduces the varieties $\mathbf K$ and $\overleftarrow{\mathbf K}$. 
In Section~\ref{suf cond} we provide a sufficient condition under which a monoid variety is HFB.
In Section~\ref{some monoids1} we provide a classification of aperiodic  monoid varieties, which implies that every variety of $\mathscr J$-trivial monoids is either HFB or contains  one of the varieties $\mathbf A^1\vee\overleftarrow{\mathbf A^1}$, $\mathbf J$, $\overleftarrow{\mathbf J}$, $\mathbf K$, $\overleftarrow{\mathbf K}$, $\mathbf L$ or $\mathbf M$ (Corollary~\ref{C: main}). 
In Section~\ref{new example} we show that $\mathbf K$ and $\overleftarrow{\mathbf K}$ are new limit varieties of monoids (Proposition~\ref{P: K is a limit variety}). 
Thus, we provide the proof of Theorem~\ref{T: main}.
Finally, Section~\ref{subvar lattice} is devoted to a description of the subvariety lattice of the limit variety $\bf K$.

\section{Preliminaries}
\label{preliminaries}

\subsection{Varieties of $\mathscr J$-trivial monoids}

The following claim is well-known but we provide its proof for the sake of completeness.

\begin{fact}
\label{F: J-trivial} 
Let $\mathbf V$ be a variety of $\mathscr J$-trivial monoids. Then $\mathbf V$ satisfies the identities
\begin{equation}
\label{id J-trivial}
x^n \approx x^{n+1}\ \text{ and }\ (xy)^n \approx (yx)^n
\end{equation} 
for some $n\ge1$.
\end{fact}

\begin{proof} 
Since the Green's relation $\mathscr J$ on any group coincides with the universal relation, all the groups of $\mathbf V$ are trivial. 
This implies that $\mathbf V$ is aperiodic and so satisfies the identity $x^n\approx x^{n+1}$ for some $n\ge1$.
 
Let $M$ be a monoid from $\mathbf V$ and $a,b\in M$. 
Then $(ab)^n = (ab)^{n+1} = a(ba)^nb$ and similarly,  $(ba)^n = b(ab)^na$. Therefore, $(ab)^n$ and $(ba)^n$ lie in the same $\mathscr J$-class of $M$. 
Since $M$ is $\mathscr J$-trivial, $(ab)^n=(ba)^n$.
It follows that $M$ and so $\mathbf V$ satisfy $(xy)^n \approx (yx)^n$.
\end{proof}

\subsection{Words, identities and Dilworth-Perkins construction}

Let $\mathfrak A$ be a countably infinite set called an \textit{alphabet}. 
As usual, let~$\mathfrak A^+$ and~$\mathfrak A^\ast$ denote the free semigroup and the free monoid over the alphabet~$\mathfrak A$, respectively. 
Elements of~$\mathfrak A$ are called \textit{letters} and elements of~$\mathfrak A^\ast$ are called \textit{words}. 
We treat the identity element of~$\mathfrak A^\ast$ as \textit{the empty word}, which is denoted by~$1$.  
Words and letters are denoted by small Latin letters. 
However, words unlike letters are written in bold. 
The \textit{content} of a word $\mathbf w$, i.e., the set of all letters occurring in $\mathbf w$, is denoted by $\con(\mathbf w)$. 
A letter is called \textit{simple} [\textit{multiple}] \textit{in a word} $\mathbf w$ if it occurs in $\mathbf w$ once [at least twice]. 
The set of all simple [multiple] letters in a word $\mathbf w$ is denoted by $\simp(\mathbf w)$ [respectively $\mul(\mathbf w)$]. 
We use $W^\le$ to denote the closure of a set of words $W\subset\mathfrak A^\ast$ under taking subwords.

The following construction was introduced by Perkins~\cite{Perkins-69} to build the first two examples of finitely generated NFB varieties of semigroups (although in essence it appears in~\cite{MH}, where it is attributed to Dilworth). 
For any set of words $W$, let $M(W)$ denote the Rees quotient monoid of $\mathfrak A^\ast$ over the ideal $\mathfrak A^\ast \setminus  {W}^{\le}$ consisting of all words that are not subwords of any word in $W$. 
Given a finite set of words $W$, $M(W)$ is a finite $\mathscr J$-trivial monoid with zero.

\subsection{Generalized Dilworth-Perkins construction}
\label{construction}

The following generalization of $M(W)$ construction was introduced by the second-named author in~\cite{Sapir-18,Sapir-21}. 
Given a congruence $\tau$ on the free monoid $\mathfrak A^\ast$ we use $\circ{_\tau}$ to denote the binary operation on the quotient monoid $\mathfrak A^\ast/\tau$.  
The elements of $\mathfrak A^\ast/\tau$ are called $\tau$-\textit{words} or $\tau$-\textit{classes} and written using lowercase letters in the typewriter style\footnote{We call the elements of $\mathfrak A^\ast/\tau$ by $\tau$-\textit{words} when we want to emphasize the relations between them and we refer to  ${\mathtt u} \in \mathfrak A^\ast/\tau$ as a $\tau$-\textit{class} when we are interested in the description of the words contained in ${\mathtt u}$.}. 
The subword relation on $\mathfrak A^\ast$ can be naturally extended to $\tau$-words as follows: given two $\tau$-words ${\mathtt u}, {\mathtt v} \in \mathfrak A^\ast/\tau$  we write ${\mathtt v} \le_\tau {\mathtt u}$ if ${\mathtt u} = {\mathtt p}\circ_\tau {\mathtt v}\circ_\tau {\mathtt s}$ for some   ${\mathtt p}, {\mathtt s} \in \mathfrak A^\ast/\tau$.
Given a set of $\tau$-words ${\mathtt W}$ we define ${\mathtt W}^{\le_\tau}$ as closure of ${\mathtt W}$ in quasi-order $\le_\tau$.
If $\mathtt W$ is a set of $\tau$-words, then $M_\tau(\mathtt W)$ denotes the Rees quotient of $\mathfrak A^\ast /\tau$ over the ideal $(\mathfrak A^\ast/ \tau) \setminus  {\mathtt W}^{\le_\tau}$.
For brevity, if $\mathtt w_1,\mathtt w_2,\dots,\mathtt w_k\in\mathfrak A^\ast/ \tau$, then we write $M_\tau(\mathtt w_1,\mathtt w_2,\dots,\mathtt w_k)$ rather than $M_\tau(\{\mathtt w_1,\mathtt w_2,\dots,\mathtt w_k\})$.
If $\tau$ is the trivial congruence on  $\mathfrak A^\ast$, then $M_\tau({\mathtt W})$ construction coincides with Dilworth-Perkins construction.

Let $\tau_1$ denote the congruence on the free monoid $\mathfrak A^\ast$ induced by the relations $a=aa$ for each $a \in\mathfrak A^\ast$.
We refine $\tau_1$ into three more congruences as follows.
Given  ${\bf u}, {\bf v} \in \mathfrak A^\ast$, we define:
\begin{itemize}
\item ${\bf u} \mathrel{\gamma} {\bf v}$ if and only if ${\bf u}  \mathrel{\tau_1} {\bf v}$ and  $\mul({\bf u}) =  \mul({\bf v})$;
\item ${\bf u}  \mathrel{\lambda} {\bf v}$ if and only if ${\bf u}  \mathrel{\gamma} {\bf v}$ and the first two occurrences of each  multiple  letter are adjacent in $\bf u$ if and only if these occurrences are  adjacent in $\bf v$;
\item ${\bf u}  \mathrel{\rho} {\bf v}$ if and only if ${\bf u}  \mathrel{\gamma} {\bf v}$ and the last two occurrences
of each multiple letter are adjacent in $\bf u$ if and only if these occurrences are  adjacent in $\bf v$.
\end{itemize}

Notice that the relations $\rho$ and $\lambda$ are dual to each other. 
If $\tau\in\{\tau_1,\gamma, \lambda, \rho\}$, then it is verified in Proposition~8.5 in~\cite{Sapir-21} that the relation $\le_\tau$ is an order on $\mathfrak A^\ast/\tau$ and given a finite set of $\tau$-words $\mathtt W$, $M_\tau({\mathtt W})$ is a finite $\mathscr J$-trivial monoid with zero.

We say that a $\tau$-word  ${\mathtt u} \in \mathfrak A^\ast/\tau$ is a $\tau$-\textit{term} for a variety $\mathbf V$ if, for any word ${\bf u}$ in ${\mathtt u}$, we have $\mathbf u\mathrel{\tau}\mathbf v$ whenever $\mathbf V$ satisfies $\mathbf u\approx\mathbf v$. 
The following lemma generalizes Lemma~3.3 in~\cite{Jackson-05}.

\begin{lemma}[\mdseries{\cite[Corollary~3.6]{Sapir-21}}]
\label{L: S_alpha in V1} 
Let  $\tau$ be either the trivial congruence or $\tau \in \{\tau_1, \gamma\}$.
Let $\mathtt W$ be a subset of $\mathfrak A^\ast / \tau$.
Then a monoid variety $\mathbf V$ contains $M_{\tau}(\mathtt W)$ if and only if every $\tau$-word in $\mathtt W$ is a $\tau$-term for $\mathbf V$.\qed
\end{lemma}

If ${\mathbf u} \in \mathfrak A^\ast$ and $x \in \con({\bf u})$, then an \textit{island} formed by $x$ in ${\bf u}$ is a maximal subword of $\bf u$, which is a power of $x$. 
For example, the word $ab^2a^5ba^3$ has three islands formed by $a$ and two islands formed by $b$.
We say that  ${\mathbf u} \in \mathfrak A^\ast$ is $2$-\textit{island-limited} if each letter forms at most $2$ islands in $\bf u$. For example, the word $asbtb^5a^7$ is $2$-island-limited.
Given  $\tau\in\{\tau_1,\gamma, \lambda, \rho\}$ we say that  ${\mathtt u} \in \mathfrak A^\ast/\tau$ is  $2$-\textit{island-limited}  if $\bf u$ is $2$-island-limited for each ${\bf u} \in {\mathtt u}$.

\begin{fact}[\mdseries{\cite[Lemma~6.3 and it's dual]{Sapir-21}}]
\label{F:2-limited}
Let $\tau \in \{\lambda, \rho\}$ and $\mathtt u$ be a $2$-island-limited $\tau$-word. 
If $\mathtt u$ is a $\tau$-term for a monoid variety $\mathbf V$, then every $\tau$-word $\mathtt v$ with $\mathtt v \le_\tau \mathtt u$ is also a $\tau$-term for $\mathbf V$.
\end{fact}

\begin{lemma}[\mdseries{\cite[Corollary~6.4 and it's dual]{Sapir-21}}]
\label{L: S_alpha in V2} 
Let $\tau \in \{\lambda, \rho\}$ and $\mathtt W$ be a set of $2$-island-limited $\tau$-words.
Then a monoid variety $\mathbf V$ contains $M_\tau(\mathtt W)$ if and only if every $\tau$-word in $\mathtt W$ is a $\tau$-term for $\mathbf V$.\qed
\end{lemma}

Let $\mathfrak B$ denote an alphabet, which consists of symbols $a^+$ for each $a\in \mathfrak A$.
If $\tau\in\{\tau_1,\gamma, \lambda, \rho\}$, then it is convenient to represent the elements of $\mathfrak A^\ast /\tau$ by words in the alphabet $\mathfrak A\cup\mathfrak B$.
For each $a \in \mathfrak A$ consider the following rewriting rules on $(\mathfrak A \cup \mathfrak B)^\ast$:
\begin{itemize}
\item $R^{\tau_1}_{a\rightarrow a^+}$ replaces $a$ by $a^+$;
\item $R^\gamma_{a\rightarrow a^+}$ replaces an occurrence of $a$ in $\bf u$ by $a^+$ only in case that either $a$ appears in $\bf u$ at least twice or $a^+$ appears in $\bf u$;
\item $R^\lambda_{a\rightarrow a^+}$ replaces an occurrence of $a$ in $\bf u$ by $a^+$ only in case that either $a$ or $a^+$ appears in $\bf u$ to the left of this occurrence of $a$;
\item $R^\rho_{a\rightarrow a^+}$ replaces an occurrence of $a$ in $\bf u$ by $a^+$ only in case that either $a$ or $a^+$ appears in $\bf u$ to the right of this occurrence of $a$;
\item $R_{a^+a^+\rightarrow a^+}$ replaces $a^+a^+$ by $a^+$;
\item $R_{aa^+\rightarrow a^+}$ replaces $aa^+$ by $a^+$;
\item $R_{a^+a\rightarrow a^+}$ replaces $a^+a$ by $a^+$.
\end{itemize}
It is proved in Lemma~8.1 in \cite{Sapir-21} that if $\tau\in\{\tau_1,\gamma, \lambda, \rho\}$, then using the rules 
\begin{equation}
\label{rewriting rules}
\{R^{\tau}_{a\rightarrow a^+}, R_{a^+a^+\rightarrow a^+}, R_{aa^+\rightarrow a^+},R_{a^+a\rightarrow a^+}\mid a \in \mathfrak A\}
\end{equation}
in any order, every word  ${\bf u} \in (\mathfrak A \cup \mathfrak B)^\ast$ can be transformed to a unique word $r_{\tau}({\bf u})$ such that none of these rules is applicable to  $r_{\tau}({\bf u})$.
Let $\mathcal R_{\tau}$ denote the set of all words in $(\mathfrak A \cup \mathfrak B)^\ast$ to which none of the rewriting rules~\eqref{rewriting rules} is applicable.
It is verified in Proposition~8.2 in~\cite{Sapir-21} that the $\tau$-words in $\mathfrak A^\ast /\tau$  can be identified with the elements of $\mathcal R_{\tau}$ such that for each pair of words $\mathbf u,\mathbf v\in \mathcal R_{\tau}$ we have
$$
\mathbf u\,{\circ{_{\tau}}}{\mathbf v} = r_{\tau}(\mathbf{uv}).
$$ 
For example, $(bta^+) \circ_\lambda (a^+b^+) = bta^+b^+$.
The next example lists all non-zero elements of the 20-element monoid $M_\lambda(bta^{+}b^{+})$.

\begin{example} 
\label{E:K}
\[
\begin{aligned}
\{bta^{+}b^{+}&\}^{\le_\lambda} = \{1,  a, a^+, b, b^+, t,\\
bt, ta, t&a^+, ab, ab^+, a^+b, a^+b^+,\\
bta&, bta^+,ta^+b, ta^+b^+,\\
&bta^+b,  bta^+b^+\}.
\end{aligned}
\]
\end{example}

If $\mathtt w_1, \mathtt w_2,\dots,\mathtt w_k \in \mathfrak A^\ast/\tau$, then we use $\mathbb M_\tau(\mathtt w_1, \mathtt w_2,\dots,\mathtt w_k)$ to denote the monoid variety generated by $M_\tau(\mathtt w_1, \mathtt w_2,\dots,\mathtt w_k)$.
If $\tau$ is the trivial congruence on $\mathfrak A^\ast$, then we write $\mathbb M(\mathtt w_1, \mathtt w_2,\dots,\mathtt w_k)$ rather than $\mathbb M_\tau(\mathtt w_1, \mathtt w_2,\dots,\mathtt w_k)$.

It turns out that every limit variety of $\mathscr J$-trivial monoids is generated by a monoid of the form $M_\tau({\mathtt W})$.
In particular,  ${\mathbf L} = \mathbb M(atbasb)$ and ${\mathbf M} = \mathbb M(abtasb, atbsab)$~\cite{Jackson-05}.
The limit variety $\mathbf J$ was introduced in~\cite{Gusev-SF} as a variety given by some infinite identity system.
According to Theorem~7.2(v) in~\cite{Sapir-21} and its dual, ${\mathbf J} = \mathbb M_\lambda(atba^{+}sb^{+})$ and $\overleftarrow{\mathbf J} = \mathbb M_\rho(a^+tb^+asb)$. 
In view of Theorem~4.3(iv) in~\cite{Sapir-21},  the limit variety $\mathbf A^1\vee\overleftarrow{\mathbf A^1}$ discovered in~\cite{Zhang-Luo} is generated by the monoid $M_{\gamma}(a^+b^+ta^+, a^+tb^+a^+)$.{\sloppy

}

Put ${\mathbf K} = \mathbb M_\lambda(bta^{+}b^{+})$.
According to Example~\ref{E:K}, the variety ${\mathbf K}$ is generated by a 20-element monoid. 
In fact, $\mathbf K$ is generated by a 12-element submonoid of $M_\lambda(bta^{+}b^{+})$ (see Remark~\ref{R: monoid in K}).
We are going to prove in Proposition~\ref{P: K is a limit variety} that ${\mathbf K}$ and $\overleftarrow{\mathbf K}$ are new limit varieties of monoids.

\subsection{Two useful facts}

If $\mathcal M$ is a monoid or a class of monoids and $\Sigma$ is an identity or a set of identities, then we write $\mathcal M \models \Sigma$ whenever $\mathcal M$ satisfies $\Sigma$. 
A word~$\mathbf w$ is an \textit{isoterm} for $\mathcal M$ if $\mathcal M$ violates any non-trivial identity of the form $\mathbf w \approx \mathbf w'$. 
For a word $\mathbf w$ and a set of letters $X \subseteq \con(\mathbf w)$, let $\mathbf w(X)$ be the word obtained from $\mathbf w$ by deleting from $\mathbf w$ all letters that are not in $X$. 
We write $\mathbf w(x_1,x_2,\dots,x_k,X)$ rather than $\mathbf w(X\cup\{x_1,x_2,\dots,x_k\})$ whenever $X\subseteq\con(\mathbf w)$ and $x_1,x_2,\dots,x_k\in\con(\mathbf w)\setminus X$.

The next two facts are well-known. 
We provide their proofs for the sake of completeness.

\begin{lemma}
\label{L: xy isoterm}
For a monoid $M$ the following are equivalent:
\begin{itemize}
\item[\textup{(i)}]$xy$ is an isoterm for $M$;
\item[\textup{(ii)}]every identity $\mathbf u \approx \mathbf v$ satisfied by $M$ has the following properties:  
$$
\mul(\mathbf u)=\mul(\mathbf v)\ \text{ and }\ \,\mathbf u(\simp(\mathbf u)) = \mathbf v(\simp(\mathbf v)).
$$
\end{itemize} 
\end{lemma}

\begin{proof} 
(i) $\rightarrow$ (ii) Clearly, $x$ is an isoterm for $M$. 
Then $\mul(\mathbf u)=\mul(\mathbf v)$ and $\simp(\mathbf u)=\simp(\mathbf v) = \{t_1, \dots, t_m\}$ for some $m \ge 0$. 
Therefore, $\mathbf u(\simp(\mathbf u)) = t_1 \dots t_m$ and  $\mathbf v(\simp(\mathbf u)) = t_{\pi1} \dots t_{\pi m}$ for some permutation $\pi$ on $\{t_1, \dots, t_m\}$. 
If $\pi$ is not the identical permutation, then $M$ satisfies $t_it_j \approx t_jt_i$ for some $1 \le i < j \le m$. 
To avoid contradiction, $\pi$ is must be the identical permutation, that is, $\mathbf u(\simp(\mathbf u)) = \mathbf v(\simp(\mathbf v))$.
Implication~(ii) $\rightarrow$ (i) is evident.
\end{proof} 

A \textit{block} of a word $\mathbf w$ is a maximal subword of $\mathbf w$ that does not contain any letters simple in $\mathbf w$. 
If $xy$ is an isoterm for a monoid $M$, then, in view of Lemma~\ref{L: xy isoterm}, every identity of $M$ is of the form
\begin{equation}
\label{..t_iu_i..=..t_iv_i..}
\mathbf u_0 \prod_{i=1}^m (t_i\mathbf u_i) \approx \mathbf v_0 \prod_{i=1}^m (t_i\mathbf v_i),
\end{equation}
where $\simp(\mathbf u)=\simp(\mathbf v) = \{t_1, \dots, t_m\}$ for some $m \ge 0$ and
$\mul(\mathbf u) = \con(\mathbf u_0  \dots \mathbf u_m)  = \con(\mathbf v_0  \dots \mathbf v_m) =\mul(\mathbf v)$.
For each $i=0,1,\dots, m$, we say the blocks $\mathbf u_i$ and $\mathbf v_i$ are \textit{corresponding}.{\sloppy

}
\begin{lemma}
\label{L: xy is not an isoterm}
Let $M$ be a monoid that satisfies $x^n \approx x^{n+1}$ for some $n\ge1$. If $xy$ is not an isoterm for $M$, then $M$ is commutative or idempotent.\qed
\end{lemma}

\begin{proof}  
If $x$ is not an isoterm for $M$, then $M$ satisfies $x \approx x^k$ for some $k>1$. 
Since the identities $x^n \approx x^{n+1}$ and $x \approx x^k$ imply $x \approx x^2$, the monoid $M$ is idempotent. 

If $x$ is  an isoterm for $M$ but $xy$ is not an isoterm for $M$, then $M$ satisfies $xy \approx yx$ and is commutative.
\end{proof}

\section{A sufficient condition under which a monoid variety is HFB}
\label{suf cond}

A monoid $M$ is said to be [\textit{hereditary}] \textit{finitely based} if $\var M$ is [H]FB. 
The variety of monoids given by an identity system $\Delta$ is denoted by $\var\,\Delta$. We fix the following set of identities for the rest of this section:
\[
\Sigma = \{xtyxsy \approx xtxyxsy,\,xtysyx \approx xtysxyx\}.
\]
The following proposition is the main result of this section and generalizes Proposition~3.1 in~\cite{Gusev-JAA}, which says that $\var \{xtyxsy \approx xtxyxsy,\, x^2y^2 \approx y^2x^2\}$ is HFB.

\begin{proposition} 
\label{P: hfb}
Let $M$ be a monoid such that $M \models \Sigma$. 
Then $M$ is HFB.
\end{proposition}

To prove Proposition~\ref{P: hfb} we need some auxiliary results. 
We will often use the following fact without references.

\begin{fact} 
\label{F: id}
The identity  
\begin{equation}
\label{xtysyx=xtysxyx} 
xtysyx \approx xtysxyx 
\end{equation}
implies the identities $xtx \approx xtx^2$ and
\begin{equation}
\label{xtysxy=xtysyxy}
xtysxy\approx xtysyxy.
\end{equation}
\end{fact}

\begin{proof}  
The identity $xtx \approx xtx^2$ can be obtained by erasing letters $s$ and $y$ from~\eqref{xtysyx=xtysxyx}. Then~\eqref{xtysxy=xtysyxy} follows from~\eqref{xtysyx=xtysxyx} because $xtysxy  \stackrel{xtx \approx xtx^2}{\approx} xtysx^2y  \stackrel{\eqref{xtysyx=xtysxyx}}{\approx}  xtys(xy)^2  \stackrel{\eqref{xtysyx=xtysxyx}}{\approx} xtysyxy$.
\end{proof}

The following lemma generalizes Lemma~3.2 in~\cite{Gusev-JAA}.

\begin{lemma} 
\label{L: insert}
Let $\mathbf w=\mathbf v_1\mathbf v_2x\mathbf v_3$, where $\mathbf v_1, \mathbf v_2,\mathbf v_3\in\mathfrak A^\ast$. 
If ${\bf v}_1$ contains $x$ and ${\bf v}_2$ does not contain any letters that are simple in $\bf w$, then $\Sigma$ implies ${\bf w}\approx {\bf v}_1 x {\bf v}_2 x {\bf v}_3$.
\end{lemma}

\begin{proof} 
In view of Fact~\ref{F: id}, the identity~\eqref{xtysyx=xtysxyx} implies~\eqref{xtysxy=xtysyxy}. 
The rest of the proof is similar to the proof of Lemma~3.2 in~\cite{Gusev-JAA}.
The only difference is that instead of~$xtysxy \approx xtysyx$ we use one of the identities~\eqref{xtysyx=xtysxyx} or~\eqref{xtysxy=xtysyxy}.
\end{proof}

Let ${\bf u} \approx {\bf v}$ be an identity such that ${\bf u}(\simp({\bf u})) = {\bf v}(\simp({\bf v}))$. 
If $x \in \mul({\bf u})$, we say that ${\bf u} \approx {\bf v}$ is \textit{$x$-well-balanced} if ${\bf u}(x, \simp({\bf u})) = {\bf v}(x, \simp({\bf u}))$. 
Following Lee~\cite{Lee-12}, we say that the identity is \textit{well-balanced} if ${\bf u} \approx {\bf v}$ is $x$-well-balanced for each $x \in \mul({\bf u})$. 
The identity ${\bf u} \approx {\bf v}$ is \textit{not well-balanced at} $x$ if the number of occurrences of $x$ in some block $\mathbf b$ of $\mathbf u$ does not equal to the number of occurrences of $x$ in the corresponding block $\mathbf b'$ of $\mathbf v$. 
If $x \in \mul({\bf u})$ and $\bf b$ is the block of $\bf u$, which contains the first occurrence of $x$, then we say that $\bf u$ is $x$-\textit{compact} if $\bf b$ contains at most two occurrences of $x$ and every other block of $\bf u$ contains at most one occurrence of $x$. 
We say that a word $\bf u$ is \textit{compact} if $\bf u$ is $x$-compact for each $x \in \mul({\bf u})$. 
We say that an identity ${\bf u} \approx {\bf v}$ is \textit{compact} if both $\bf u$ and $\bf v$ are compact and ${\bf u}(\simp({\bf u})) = {\bf v}(\simp({\bf v}))$. 
A word that contains at most one multiple letter is called \textit{almost-linear}. 
An identity ${\bf u} \approx {\bf v}$ is called \textit{almost-linear} if both $\bf u$ and $\bf v$ are almost-linear. 
Given a monoid $M$ we use $\Gamma(M)$ to denote the set of all almost-linear identities of $M$.

\begin{lemma} 
\label{L: compact} 
Let $M$ be a monoid such that $xy$ is an isoterm for $M$ and $M \models \Sigma$. 
Then every identity of $M$ can be derived from a compact well-balanced identity of $M$ together with $\Sigma \cup \Gamma(M)$.
\end{lemma}

\begin{proof} 
Let $\mathbf u\approx\mathbf v$ be an identity of $M$. In view of Lemma~\ref{L: xy isoterm}, $\mul(\mathbf u)=\mul(\mathbf v)$ and $\mathbf u(\simp(\mathbf u))=\mathbf v(\simp(\mathbf v))$. 
Let ${\bf u} \approx {\bf v}$ be not well-balanced at precisely $k$ different letters. We will use induction on $k$.

{\bf Induction base}: $k=0$. 
Then ${\bf u} \approx {\bf v}$ is well-balanced. 
In view of Lemma~\ref{L: insert}, we can remove some occurrences of letters in $\mathbf u$ and $\mathbf v$ and obtain the words $\mathbf u^\ast$ and $\mathbf v^\ast$ such that the identity $\mathbf u^\ast\approx\mathbf v^\ast$ is compact and well-balanced and ${\bf u} \approx {\bf v}$ is equivalent modulo $\Sigma$ to $\mathbf u^\ast\approx\mathbf v^\ast$, we are done.

{\bf Induction step}: $k>0$. 
Then ${\bf u}' = {\bf u}(x, \simp({\bf u})) \ne {\bf v}(x, \simp({\bf v}))={\bf v}'$ for some $x \in \mul({\bf u})$.

Suppose that every block of $\bf u$ contains at most one occurrence of $x$. 
Then we apply the identity ${\bf u}' \approx {\bf v}'$ to ${\bf u}$ and obtain a word $\bf w$ such that the identity ${\bf w} \approx {\bf v}$ is $x$-well-balanced. 
By the induction assumption, the identity ${\bf w} \approx {\bf v}$ can be derived from a compact well-balanced identity $\sigma$ of $M$ together with $\Sigma \cup \Gamma(M)$. 
Then ${\bf u} \approx {\bf v}$ follows from $\{\sigma\}\cup\Sigma\cup\Gamma(M)$ because ${\bf u}' \approx {\bf v}'\in\Gamma(M)$, and we are done. 
By a similar argument we can show that if every block of $\bf v$ contains at most one occurrence of $x$, then ${\bf u} \approx {\bf v}$ follows from a compact well-balanced identity of $M$ together with $\Sigma \cup \Gamma(M)$.
So, by Lemma~\ref{L: insert}, we may assume that some block $\bf a$ of $\bf u$ contains the first and second occurrences of $x$ and some block $\bf b$ of $\bf v$ also contains the first and second occurrences of $x$. 
Two cases are possible.

{\bf Case 1}: the blocks $\mathbf a$ and $\mathbf b$ are corresponding.  
Lemma~\ref{L: insert} allows us to remove some occurrences of letters in $\mathbf u$ and $\mathbf v$ and obtain the words $\mathbf u^\ast$ and $\mathbf v^\ast$ such that the identity $\mathbf u^\ast\approx\mathbf v^\ast$ is compact well-balanced and ${\bf u} \approx {\bf v}$ is equivalent modulo $\Sigma$ to $\mathbf u^\ast\approx\mathbf v^\ast$.
Therefore, we may assume that $\mathbf u$ and $\mathbf v$ are compact.
Then, by symmetry, we may assume that some block $\mathbf b'$ of $\mathbf v$ contains $x$ but the corresponding block $\bf a'$ of $\bf u$ to $\mathbf b'$ does not contain $x$. 
We may choose $\mathbf a'$ and $\mathbf b'$ so that they are the rightmost blocks in the words $\mathbf u'$ and $\mathbf v'$ with such a property. 
Let $t$ be simple letter that is next to the block $\mathbf a'$ on the left of it in $\mathbf u$. 

{\bf Subcase 1.1}: the last occurrence of $x$ precedes the block $\mathbf a'$ in $\mathbf u$.  Then the identity $\mathbf u(x,t)\approx\mathbf v(x,t)$ is equivalent modulo $\Sigma$ to 
\begin{equation}
\label{xxt=xxtx}
x^2t \approx x^2tx,
\end{equation}
whence $M$ satisfies
\[
xtxs \stackrel{xtx \approx xtx^2}\approx xtx^2s  \stackrel{\eqref{xxt=xxtx}}\approx xtx^2sx \stackrel{xtx \approx xtx^2}\approx xtxsx.
\]
Then we apply the identity $xtxs \approx xtxsx$ to ${\bf u}$ and obtain a word $\bf w$ such that the identity ${\bf w} \approx {\bf v}$ is $x$-well-balanced. 
By the induction assumption, the identity ${\bf w} \approx {\bf v}$ can be derived from a compact well-balanced identity $\sigma$ of $M$ together with $\Sigma \cup \Gamma(M)$. 
Then ${\bf u} \approx {\bf v}$ follows from $\{\sigma\}\cup\Sigma\cup\Gamma(M)$ because $xtxs \approx xtxsx\in\Gamma(M)$, and we are done. 

{\bf Subcase 1.2}: the last occurrence of $x$ is preceded by the block $\mathbf a'$ in $\mathbf u$. 
Let $s$ be simple letter that is next to the blocks $\mathbf a'$ on the right of it in $\mathbf u$.
Then the identity $\mathbf u(x,t,s)\approx\mathbf v(x,t,s)$ is equivalent modulo $\Sigma$ to 
\begin{equation}
\label{xxtxsx=xxtsx}
x^2tsx \approx x^2txsx,
\end{equation}
whence $M$ satisfies
\[
xzxtxsx \stackrel{xtx \approx xtx^2}\approx xzx^2txsx  \stackrel{\eqref{xxtxsx=xxtsx}} \approx xzx^2tsx \stackrel{xtx \approx xtx^2}\approx xzxtsx.
\]
Then we apply the identity 
\begin{equation}
\label{xzxtxsx=xzxtsx}
xzxtxsx \approx xzxtsx
\end{equation}
to ${\bf u}$ and obtain a word $\bf w$ such that the identity ${\bf w} \approx {\bf v}$ is $x$-well-balanced. 
By the induction assumption, the identity ${\bf w} \approx {\bf v}$ can be derived from a compact well-balanced identity $\sigma$ of $M$ together with $\Sigma \cup \Gamma(M)$. 
Then ${\bf u} \approx {\bf v}$ follows from $\{\sigma\}\cup\Sigma\cup\Gamma(M)$ because $\eqref{xzxtxsx=xzxtsx}\in\Gamma(M)$, and we are done. 

{\bf Case 2}: the blocks $\mathbf a$ and $\mathbf b$ are not corresponding. 
Let $\mathbf a'$ be the corresponding block to $\mathbf b$. 
We may assume without loss of generality that the block $\mathbf a'$ precedes the block $\mathbf a$ in $\mathbf u$. 
Let $t$ be simple letter that is next to the block $\mathbf a'$ on the right of it in $\mathbf u$.
Clearly, the identity $\mathbf u(x,t)x\approx\mathbf v(x,t)x$ is equivalent modulo $\Sigma$ to 
\begin{equation}
\label{txx=xxtx}
tx^2\approx x^2tx,
\end{equation}
whence $M$ satisfies
\[
txsx \stackrel{xtx \approx xtx^2}\approx txsx^2  \stackrel{\eqref{txx=xxtx}}\approx x^2txsx^2 \stackrel{xtx \approx xtx^2}\approx x^2txsx.
\]
Then we apply the identity $txsx \approx x^2txsx$ to ${\bf u}$ and obtain a word $\bf w$ such that the corresponding block of $\mathbf w$ to the block $\mathbf b$ of $\mathbf v$ contains the first and second occurrences of $x$ in $\mathbf w$. 
By Case~1, the identity ${\bf w} \approx {\bf v}$ can be derived from a compact well-balanced identity $\sigma$ of $M$ together with $\Sigma \cup \Gamma(M)$. 
Then ${\bf u} \approx {\bf v}$ follows from $\{\sigma\}\cup\Sigma\cup\Gamma(M)$ because $txsx \approx x^2txsx\in\Gamma(M)$, and we are done. 
\end{proof}

The expression $_{i\mathbf w}x$ means the $i$th occurrence of a letter $x$ in a word $\mathbf w$. 
We say that a pair of occurrences $\{_i x, {_j y}\}$ of letters $x$ and $y$  in a well-balanced identity ${\bf u} \approx {\bf v}$ is \textit{critical} if $\bf u$ contains ${_{i\mathbf u} x}\,{_{j\mathbf u} y}$ as a subword and ${_{j\mathbf v} y}$ precedes ${_{i\mathbf v} x}$ in $\bf v$.
Let $\bf w$ denote the result of replacing ${_{i\mathbf u} x}\,{_{j\mathbf u} y}$ by ${_{j\mathbf u} y}\,{_{i\mathbf u} x}$ in $\bf u$. 
Given a set of identities $\Delta$ and a well-balanced identity ${\bf u} \approx {\bf v}$, we say that the critical pair $\{_i x, {_j y}\}$ is $\Delta$-removable in ${\bf u} \approx {\bf v}$ if $\Delta$ implies ${\bf u} \approx {\bf w}$.

The following special case of Lemma~3.4 in~\cite{Sapir-15} describes the standard method of deriving identities by removing critical pairs.  
This method traces back to the articles~\cite{Jackson-Sapir,Sapir-00}. 

\begin{lemma} 
\label{L: fblemma} 
Let $M$ be a monoid and $\Delta$ be a set of identities. 
Suppose that each critical pair in every  compact well-balanced identity of $M$ is $\Delta$-removable. 
Then every compact well-balanced identity of $M$ can be derived from $\Delta$.\qed
\end{lemma}

Let $\Phi$ denote the set of compact well-balanced identities with two multiple letters of the form
${\bf p}{\bf c} \approx {\bf q}{\bf c}$ or $ys{\bf p}{\bf c} \approx ys{\bf q}{\bf c}$, where ${\bf p}$ and $\mathbf q$ are words in $\{x,y\}^+$ with the property that both $x$ and $y$ appear at least once and at most twice in both ${\bf p}$ and $\bf q$ and occur in $\mathbf p$ the same number of times as in $\mathbf q$, and $\bf c$ be either the empty word or a word from the set
\begin{equation} 
\label{C} 
\{txy,\,\prod_{\ell=1}^k(t_\ell\mathbf e_\ell) \mid k \ge 1, \mathbf e_1,\mathbf e_2,\dots,\mathbf e_k\in\{1,x,y\}\}.
\end{equation}
Notice that if ${\bf u} \approx {\bf v}$ is an identity from $\Phi$ then the difference between $\bf u$ and $\bf v$ is only in the corresponding  blocks $\bf p$ and $\bf q$. 
Also notice that at most two blocks of $\bf u$ (and $\bf v$) contain both $x$ and $y$. Moreover, if two blocks of $\bf u$ contain both $x$ and $y$ then $\bf u$ has at most three blocks and the last block of $\bf u$ is $xy$. 
For example, the set $\Phi$ contains $ysxy^2txy \approx ysyxytxy$. 
For more examples of identities in $\Phi$ we refer the reader to  look at the cluster of seven identities used in the proof of Proposition~\ref{P: L(K)} below.

Given a monoid $M$ we use $\Phi(M)$ to denote the set of those identities from $\Phi$ that hold in $M$.

\begin{proof}[Proof of Proposition~\ref{P: hfb}]  
If $xy$ is not an isoterm for $M$, then $M$ is either commutative or idempotent by Lemma~\ref{L: xy is not an isoterm}. 
Then $M$ is FB because each commutative monoid~\cite{Head-68} and each idempotent monoid~\cite{Wismath-86} are FB. 
So, we may assume that $xy$ is an isoterm for $M$. 
According to Lemma~\ref{L: compact}, every identity of $M$ can be derived from a compact well-balanced identity of $M$ together with $\Sigma \cup \Gamma(M)$. 
By the result of Volkov~\cite[Corollary~2]{Volkov-90}, every set of almost-linear identities gives a FB variety. 
Thus, $\var \Gamma(M)$ is FB. 
It is easy to see that every subset of $\Phi$ gives a FB subvariety within $\var\Sigma$. 
So, it suffices to prove that every compact well-balanced identity ${\bf u} \approx {\bf v}$ of $M$ follows from $\Sigma\cup\Phi(M)$. 

Let $\{_i x, {_j y}\}$ be a critical pair in ${\bf u} \approx {\bf v}$ and $\bf w$ denote the result of replacing ${_{i\mathbf u} x}\,{_{j\mathbf u} y}$ by ${_{j\mathbf u} y}\,{_{i\mathbf u} x}$ in $\bf u$. 
We consider four cases and show that in every case the identity ${\bf u} \approx {\bf v}$ together with $\Sigma$  imply some identity in $\Phi(M)$ which (together with $\Sigma$) can be used to ``remove the critical pair $\{_i x, {_j y}\}$", that is, to derive the identity ${\bf u} \approx {\bf w}$. 

Since ${\bf u} \approx {\bf v}$ is well-balanced, we have $\mathbf u(\simp({\bf u})) = \mathbf v(\simp({\bf v}))$ and both $x$ and $y$ are multiple letters. 
Let $\bf a$ denote the block of $\bf u$, which contains the critical pair $\{_i x, {_j y}\}$ and $\mathbf a'$ denote the corresponding block of $\mathbf v$ to $\mathbf a$. 
Then $\mathbf a=\mathbf a_1\,{_{i\mathbf u} x}\,{_{j\mathbf u} y}\,\mathbf a_2$, $\mathbf u=\mathbf u'\mathbf a\mathbf u''$ and $\mathbf v=\mathbf v'\mathbf a'\mathbf v''$ for some $\mathbf a_1,\mathbf a_2,\mathbf u',\mathbf u'',\mathbf v',\mathbf v''\in\mathfrak A^\ast$.

{\bf Case 1}: block $\bf a$ contains neither the first occurrence of $x$ nor the first occurrence of $y$. 
In this case, $\mathbf a(x,y) = xy$ and $\mathbf a'(x,y) = yx$ because $\mathbf u\approx \mathbf v$ is compact and well-balanced. 
Two subcases are possible.

{\bf Subcase 1.1}: some block $\mathbf b$ of ${\bf u}''$ contains both an occurrence of $x$ and an occurrence of $y$. 
Let $\mathbf u''=\mathbf u_1\mathbf b\mathbf u_2$ for some $\mathbf u_1,\mathbf u_2\in\mathfrak A^\ast$. 
Let $t$ be the simple letter that is immediately to the right of the block $\bf a$. Then ${\bf u} \approx {\bf v}$ implies
\[ 
{\bf u}'(x,y)xyt{\bf u}''(x,y) \approx {\bf v}'(x,y)yxt{\bf v}''(x,y)
\]
and so
\[ 
{\bf u}'(x,y)xyt{\bf u}''(x,y)xy \approx {\bf v}'(x,y)yxt{\bf v}''(x,y)xy.
\]
The last identity is equivalent modulo $\Sigma$ to
\begin{equation} 
\label{pxytxy=qyxtxy} 
{\bf p}xytxy \approx {\bf q}yxtxy,
\end{equation}
where $\mathbf p,\mathbf q\in\{xy,yx\}$ because we can remove all non-first occurrences of $x$ and $y$ in ${\bf u}'(x,y)$ and ${\bf v}'(x,y)$ and  the words ${\bf u}''(x,y)$ and ${\bf v}''(x,y)$ by Lemma~\ref{L: insert}. 
Clearly, $\eqref{pxytxy=qyxtxy}\in \Phi(M)$. 
Then $M$ satisfies
\[
\mathbf u\stackrel{\text{Lemma~\ref{L: insert}}}\approx \mathbf u'\mathbf a_1\mathbf pxy\mathbf a_2\mathbf u_1xy\mathbf b\mathbf u_2\stackrel{\eqref{pxytxy=qyxtxy}}\approx \mathbf u'\mathbf a_1\mathbf qyx\mathbf a_2\mathbf u_1xy\mathbf b\mathbf u_2\stackrel{\text{Lemma~\ref{L: insert}}}\approx \mathbf w,
\]
and we are done.

{\bf Subcase 1.2}: no block of $\mathbf u''$ contains both an occurrence of $x$ and an occurrence of $y$. 
Let $\{t_1, t_2,\dots, t_k\}$ be the possibly empty set of simple letters of $\bf u$ in $\mathbf u''$. Then ${\bf u} \approx {\bf v}$ implies
\begin{equation}
\label{subst identity}
{\bf u}'(x,y)xy\mathbf u''(x,y,t_1,t_2,\dots,t_k) \approx {\bf v}'(x,y)yx\mathbf v''(x,y,t_1,t_2,\dots,t_k).
\end{equation}
Since the identity ${\bf u} \approx {\bf v}$ is well-balanced and compact,
\begin{equation}
\label{u''=v''=e}
\mathbf u''(x,y,t_1,t_2,\dots,t_k)=\mathbf v''(x,y,t_1,t_2,\dots,t_k)=\prod_{\ell=1}^k(t_\ell\mathbf e_\ell),
\end{equation}
where $\mathbf e_1,\mathbf e_2,\dots,\mathbf e_k\in\{1,x,y\}$. 
The identity~\eqref{subst identity} is equivalent modulo $\Sigma$ to
\begin{equation} 
\label{pxyte=qyxte} 
{\bf p}xy\prod_{\ell=1}^k(t_\ell\mathbf e_\ell) \approx {\bf q}yx\prod_{\ell=1}^k(t_\ell\mathbf e_\ell),
\end{equation}
where $\mathbf p,\mathbf q\in\{xy,yx\}$ because we can remove all non-first occurrences of $x$ and $y$ in ${\bf u}'(x,y)$ and ${\bf v}'(x,y)$ by Lemma~\ref{L: insert}. 
Clearly, $\eqref{pxyte=qyxte}\in \Phi(M)$. 
Then $M$ satisfies
\[
\mathbf u\stackrel{\text{Lemma~\ref{L: insert}}}\approx \mathbf u'\mathbf a_1\mathbf pxy\mathbf a_2\mathbf u''\stackrel{\eqref{pxyte=qyxte}}\approx \mathbf u'\mathbf a_1\mathbf qyx\mathbf a_2\mathbf u''\stackrel{\text{Lemma~\ref{L: insert}}}\approx \mathbf w,
\]
and we are done.

{\bf Case 2}:  block $\bf a$ contains the first occurrence of $x$ and the first occurrence of $y$.
In this case, both $x$ and $y$ occur at least once and at most twice in ${\bf a}$ because $\bf u$ is compact. 
Let $\mathbf a(x,y) = \mathbf p\,{_{i\mathbf u}x}\,{_{j\mathbf u}y}\mathbf q$ and $\mathbf a'(x,y) = \mathbf p'\,{_{j\mathbf v}y}\,\mathbf q'\,{_{i\mathbf v}x}\mathbf r'$. 
Two subcases are possible.

{\bf Subcase 2.1}: some block $\mathbf b$ of ${\bf u}''$ contains both an occurrence of $x$ and an occurrence of $y$. 
Let $\mathbf u''=\mathbf u_1\mathbf b\mathbf u_2$ for some $\mathbf u_1,\mathbf u_2\in\mathfrak A^\ast$. 
Let $t$ be the simple letter that is immediately to the right of the block $\bf a$. 
Then  ${\bf u} \approx {\bf v}$  implies
\[
\mathbf pxy\mathbf qt{\bf u}''(x,y) \approx \mathbf p'y\mathbf q'x\mathbf r't{\bf v}''(x,y)
\]
and so
\[ 
\mathbf pxy\mathbf qt{\bf u}''(x,y)xy \approx \mathbf p'y\mathbf q'x\mathbf r't{\bf v}''(x,y)xy.
\]
The last identity is equivalent modulo $\Sigma$ to
\begin{equation} 
\label{pxyqtxy=p'yq'xr'txy} 
\mathbf pxy\mathbf qtxy \approx \mathbf p'y\mathbf q'x\mathbf r'txy
\end{equation}
because we can remove  the words ${\bf u}''(x,y)$ and ${\bf v}''(x,y)$ by Lemma~\ref{L: insert}. 
Clearly, $\eqref{pxyqtxy=p'yq'xr'txy} \in \Phi(M)$. 
Then $M$ satisfies
\[
\mathbf u\stackrel{\text{Lemma~\ref{L: insert}}}\approx \mathbf u'\mathbf a_1\mathbf pxy\mathbf q\mathbf a_2\mathbf u_1xy\mathbf b\mathbf u_2\stackrel{\eqref{pxyqtxy=p'yq'xr'txy}}\approx \mathbf u'\mathbf a_1\mathbf p'y\mathbf q'x\mathbf r'\mathbf a_2\mathbf u_1xy\mathbf b\mathbf u_2.
\]
Evidently, if $\mathbf p'=1$, then $\con(\mathbf p')\subseteq\con(\mathbf a_1)$. 
Let now $\mathbf p'\ne1$. 
If $x\in \con(\mathbf p')$, then $i=2$ and so $x\in\con(\mathbf p)$.
If $y\in \con(\mathbf p')$, then $j=2$ and so $y\in\con(\mathbf p)$.
We see that $\con(\mathbf p')\subseteq\con(\mathbf p)\subseteq\con(\mathbf a_1)$ in either case.
Analogously, $\con(\mathbf r')\subseteq\con(\mathbf a_2)$. 
Further, if $x\in\con(\mathbf q')$, then $i=2$ and so $x\in\con(\mathbf a_1)$, and if $y\in\con(\mathbf q')$, then $j=1$ and so $y\in\con(\mathbf a_2)$ because $\mathbf u$ is compact. 
This implies that $M$ satisfies $\mathbf u'\mathbf a_1\mathbf p'y\mathbf q'x\mathbf r'\mathbf a_2\mathbf u_1xy\mathbf b\mathbf u_2\approx \mathbf w$ by Lemma~\ref{L: insert}, and we are done.

{\bf Subcase 2.2}: no block of $\mathbf u''$ contains both an occurrence of $x$ and an occurrence of $y$. 
Let $\{t_1, t_2,\dots, t_k\}$ be the possibly empty set of simple letters of $\bf u$ in $\mathbf u''$.
Since the identity ${\bf u} \approx {\bf v}$ is well-balanced and compact, the equality~\eqref{u''=v''=e} holds, where $\mathbf e_1,\mathbf e_2,\dots,\mathbf e_k\in\{1,x,y\}$. 
Then ${\bf u} \approx {\bf v}$ implies
\begin{equation} 
\label{pxyqte=p'yq'xr'te} 
{\bf p}xy{\bf q}\prod_{\ell=1}^k(t_\ell\mathbf e_\ell) \approx {\bf p}'y{\bf q}'x{\bf r}'\prod_{\ell=1}^k(t_\ell\mathbf e_\ell).
\end{equation}
Clearly, $\eqref{pxyqte=p'yq'xr'te}\in \Phi(M)$. 
Then $M$ satisfies
\[
\mathbf u\stackrel{\text{Lemma~\ref{L: insert}}}\approx \mathbf u'\mathbf a_1\mathbf pxy\mathbf q\mathbf a_2\mathbf u''\stackrel{\eqref{pxyqte=p'yq'xr'te}}\approx \mathbf u'\mathbf a_1\mathbf p'y\mathbf q'x\mathbf r'\mathbf a_2\mathbf u''.
\]
Evidently, if $\mathbf p'=1$, then $\con(\mathbf p')\subseteq\con(\mathbf a_1)$. 
Let now $\mathbf p'\ne1$. 
If $x\in \con(\mathbf p')$, then $i=2$ and so $x\in\con(\mathbf p)$.
If $y\in \con(\mathbf p')$, then $j=2$ and so $y\in\con(\mathbf p)$.
We see that $\con(\mathbf p')\subseteq\con(\mathbf p)\subseteq\con(\mathbf a_1)$ in either case.
Analogously, $\con(\mathbf r')\subseteq\con(\mathbf a_2)$. 
Further, if $x\in\con(\mathbf q')$, then $i=2$ and so $x\in\con(\mathbf a_1)$, and if $y\in\con(\mathbf q')$, then $j=1$ and so $y\in\con(\mathbf a_2)$ because $\mathbf u$ is compact. 
This implies that $M$ satisfies $\mathbf u'\mathbf a_1\mathbf p'y\mathbf q'x\mathbf r'\mathbf a_2\mathbf u''\approx \mathbf w$ by Lemma~\ref{L: insert}, and we are done.

{\bf Case 3}: block $\bf a$ contains the first occurrence of $x$ but does not contain the first occurrence of $y$. 
In this case, letter $x$ appears at most twice in ${\bf a}(x,y)$ but letter $y$ only once because $\bf u$ is compact. 
Let $\mathbf a(x,y) = \mathbf p\,{_{i\mathbf u}x}\,{_{j\mathbf u}y}\mathbf q$ and $\mathbf a'(x,y) = \mathbf p'\,{_{j\mathbf v}y}\,\mathbf q'\,{_{i\mathbf v}x}\mathbf r'$. 
Let $s$ be a simple letter that is next to the block $\mathbf a$ on the left of it. 
Two subcases are possible.

{\bf Subcase 3.1}: some block $\mathbf b$ of ${\bf u}''$ contains both an occurrence of $x$ and an occurrence of $y$. 
Let $\mathbf u''=\mathbf u_1\mathbf b\mathbf u_2$ for some $\mathbf u_1,\mathbf u_2\in\mathfrak A^\ast$. 
Let $t$ be the simple letter that is immediately to the right of the block $\bf a$.
Then ${\bf u} \approx {\bf v}$ implies
\[
y^ms\mathbf pxy\mathbf qt{\bf u}''(x,y) \approx y^ms\mathbf p'y\mathbf q'x\mathbf r't{\bf v}''(x,y)
\]
and so
\[
y^ms\mathbf pxy\mathbf qt{\bf u}''(x,y)xy \approx y^ms\mathbf p'y\mathbf q'x\mathbf r't{\bf v}''(x,y)xy,
\]
where $m$ is the number of occurrences of $y$ in $\mathbf u'$. 
The last identity is equivalent modulo $\Sigma$ to
\begin{equation} 
\label{y^mpxyqtxy=y^mp'yq'xr'txy} 
y^ms\mathbf pxy\mathbf qtxy \approx y^ms\mathbf p'y\mathbf q'x\mathbf r'txy
\end{equation}
because we can remove the words ${\bf u}''(x,y)$ and ${\bf v}''(x,y)$ by Lemma~\ref{L: insert}. 
If $m>1$, then $M$ satisfies
\[
ys\mathbf pxy\mathbf qtxy\stackrel{\text{Lemma~\ref{L: insert}}}\approx ysy^m\mathbf pxy\mathbf qtxy\stackrel{\eqref{y^mpxyqtxy=y^mp'yq'xr'txy}}\approx ysy^m\mathbf p'y\mathbf q'x\mathbf r'txy\stackrel{\text{Lemma~\ref{L: insert}}}\approx ys\mathbf p'y\mathbf q'x\mathbf r'txy.
\]
So, in either case the identity
\begin{equation}
\label{ypxyqtxy=yp'yq'xr'txy} 
ys\mathbf pxy\mathbf qtxy \approx ys\mathbf p'y\mathbf q'x\mathbf r'txy
\end{equation}
holds in $M$. 
Clearly, $\eqref{ypxyqtxy=yp'yq'xr'txy} \in \Phi(M)$. 
Then $M$ satisfies
\[
\mathbf u\stackrel{\text{Lemma~\ref{L: insert}}}\approx \mathbf u'\mathbf a_1\mathbf pxy\mathbf q\mathbf a_2\mathbf u_1xy\mathbf b\mathbf u_2\stackrel{\eqref{ypxyqtxy=yp'yq'xr'txy}}\approx \mathbf u'\mathbf a_1\mathbf p'y\mathbf q'x\mathbf r'\mathbf a_2\mathbf u_1xy\mathbf b\mathbf u_2.
\]

Since $\mathbf p'\mathbf q'\mathbf r'\in\{1,x\}$, if $\mathbf p'\mathbf q'\ne1$, then $i=2$, whence $\con(\mathbf p'\mathbf q')\subseteq\con(\mathbf p)$ and so $\con(\mathbf p'\mathbf q')\subseteq\con(\mathbf a_1)$. 
Analogously, $\con(\mathbf r')\subseteq\con(\mathbf a_2)$. 
This implies that $M$ satisfies $\mathbf u'\mathbf a_1\mathbf p'y\mathbf q'x\mathbf r'\mathbf a_2\mathbf u_1xy\mathbf b\mathbf u_2\approx \mathbf w$ by Lemma~\ref{L: insert}, and we are done.

{\bf Subcase 3.2}: no block of $\mathbf u''$ contains both an occurrence of $x$ and an occurrence of $y$. 
Let $\{t_1, t_2,\dots, t_k\}$ be the possibly empty set of simple letters of $\bf u$ in $\mathbf u''$. Since the identity ${\bf u} \approx {\bf v}$ is well-balanced and compact, the equality~\eqref{u''=v''=e} holds, where $\mathbf e_1,\mathbf e_2,\dots,\mathbf e_k\in\{1,x,y\}$. 
Then ${\bf u} \approx {\bf v}$ implies
\begin{equation} 
\label{y^mspxyqte=y^msp'yq'xr'te} 
y^ms{\bf p}xy{\bf q}\mathbf c \approx y^ms{\bf p}'y{\bf q}'x{\bf r}'\mathbf c,
\end{equation}
where $m$ is the number of occurrences of $y$ in $\mathbf u'$ and $\mathbf c=\prod_{\ell=1}^k(t_\ell\mathbf e_\ell)$. 
If $m>1$, then $M$ satisfies
\[
ys\mathbf pxy\mathbf q\mathbf c\stackrel{\text{Lemma~\ref{L: insert}}}\approx ysy^m\mathbf pxy\mathbf q\mathbf c\stackrel{\eqref{y^mspxyqte=y^msp'yq'xr'te}}\approx ysy^m\mathbf p'y\mathbf q'x\mathbf r'\mathbf c\stackrel{\text{Lemma~\ref{L: insert}}}\approx ys\mathbf p'y\mathbf q'x\mathbf r'\mathbf c.
\]
So, in either case the identity
\begin{equation}
\label{yspxyqc=ysp'yq'xr'c} 
ys\mathbf pxy\mathbf q\mathbf c \approx ys\mathbf p'y\mathbf q'x\mathbf r'\mathbf c
\end{equation}
holds in $M$.  
Clearly, $\eqref{yspxyqc=ysp'yq'xr'c}\in \Phi(M)$. Then $M$ satisfies
\[
\mathbf u\stackrel{\text{Lemma~\ref{L: insert}}}\approx \mathbf u'\mathbf a_1\mathbf pxy\mathbf q\mathbf a_2\mathbf u''\stackrel{\eqref{yspxyqc=ysp'yq'xr'c}}\approx \mathbf u'\mathbf a_1\mathbf p'y\mathbf q'x\mathbf r'\mathbf a_2\mathbf u''.
\]
Since $\mathbf p'\mathbf q'\mathbf r'\in\{1,x\}$, if $\mathbf p'\mathbf q'\ne1$, then $i=2$, whence $\con(\mathbf p'\mathbf q')\subseteq\con(\mathbf p)$ and so $\con(\mathbf p'\mathbf q')\subseteq\con(\mathbf a_1)$. 
Analogously, $\con(\mathbf r')\subseteq\con(\mathbf a_2)$. 
This implies that $M$ satisfies $\mathbf u'\mathbf a_1\mathbf p'y\mathbf q'x\mathbf r'\mathbf a_2\mathbf u''\approx \mathbf w$ by Lemma~\ref{L: insert}, and we are done.

{\bf Case 4}: block $\bf a$ contains the first occurrence of $y$ but does not contain the first occurrence of $x$. This case is similar to Case 3.

So, we have proved that the critical pair $\{_i x, {_j y}\}$ is $(\Sigma\cup\Phi(M))$-removable in ${\bf u} \approx {\bf v}$. 
Now Lemma~\ref{L: fblemma} applies, with the conclusion that every compact well-balanced identity $\mathbf u\approx \mathbf v$ of $M$ can be derived from $\Sigma\cup\Phi(M)$, and we are done.
\end{proof}

\begin{corollary}
\label{C: hfb1}  
Any monoid $M$ that satisfies $xtx \approx xtx^2$ and 
\begin{equation}
\label{xyytx=xyyxtx}
xy^2tx \approx xy^2xtx
\end{equation}
is HFB.
\end{corollary}

\begin{proof} 
The identities
\[
\begin{aligned}
&xtyxsy  \stackrel{xtx \approx xtx^2}\approx xtyx^2sy \stackrel{\eqref{xyytx=xyyxtx}}{\approx} xtyx^2ysy\stackrel{xtx \approx xtx^2}\approx xtyxysy,\\
&xysytx  \stackrel{xtx \approx x^2tx}\approx xy^2sytx \stackrel{\eqref{xyytx=xyyxtx}}{\approx} xy^2xsytx\stackrel{xtx \approx x^2tx}{\approx} xyxsytx
\end{aligned}
\]
hold in $M$. 
Therefore, $M$ is HFB by the dual to Proposition~\ref{P: hfb}.
\end{proof}

Put 
\[
E = \langle a, b, c \mid a^2 = ab = 0, ba = ca = a, b^2 = bc = b, c^2= cb = c \rangle=\{a,b,c,ac,0\}.
\]
The monoid $E^1$ was first investigated in Lee and Li \cite[Section~14]{Lee-Li}, where it was shown to be finitely based by $\{xtx \approx xtx^2 \approx x^2tx,\, xy^2x \approx x^2y^2\}$. 
Let $\overline{A}^1$ denote the dual of the monoid $A^1$ (see Section~\ref{sec: intr}).

\begin{example} 
\label{E: EB}
The monoid $\overline{A}^1\times E^1$ is HFB.
\end{example}

\begin{proof} 
It follows from~\cite[Section~14]{Lee-Li} that $E^1 \models \{xtx \approx xtx^2,\, xy^2tx \approx xy^2xtx\}$. 
According to~\cite[p.~15]{Zhang-Luo}, we also have $\overline{A}^1\models \{xtx \approx xtx^2,\, xy^2tx \approx xy^2xtx\}$. 
Therefore, $\overline{A}^1\times E^1$ is HFB by Corollary~\ref{C: hfb1}.
\end{proof} 

Example~\ref{E: EB} generalizes the result from~\cite{Jackson-Lee} that $E^1$ is HFB and the result from~\cite{Zhang-Luo} that $A^1$ is HFB.

\section{Classification of varieties of aperiodic monoids}
\label{some monoids1}

Recall  from Section~\ref{preliminaries} that $\mathbf L=\mathbb M(xtxysy)$ and $\mathbf M=\mathbb M(xytxsy,xsytxy)$
are  limit varieties of monoids \cite{Jackson-05}.
The following lemma is a combination of~\cite[Theorem~3.2]{Lee-12} and~\cite[Lemma~2.1]{Gusev-JAA}.

\begin{sortlemma}
\label{SL: 1}
Let $\mathbf V$ be a variety of aperiodic monoids. Then either $\mathbf V$ is HFB or one of the following holds:
\begin{itemize}
\item[\textup{(i)}] $\mathbf V$ contains either $\mathbf L$ or $\mathbf M$;
\item[\textup{(ii)}] $\mathbf V$ satisfies either $xtx \approx xtx^2$ or $xtx \approx x^2tx$.\qed
\end{itemize}
\end{sortlemma}

The goal of this section is to refine Sorting Lemma~\ref{SL: 1} (see Sorting Lemma~\ref{SL: 2} below).

We note that some varieties can be generated by monoids of the form $M_\tau(\mathtt u)$ for several congruences $\tau$. 
For example, $M_\gamma(a^+t) \cong M_\lambda(a^+t) \cong M_\rho(a^+t)$ \cite[Fact~6.2]{Sapir-21}. 
In such a case, we choose the coarsest congruence $\tau$ to identify the monoid variety.

\begin{lemma}
\label{L: does not contain E}
Let $\mathbf V$ be a monoid variety such that $t$ is an isoterm for $\mathbf V$ and $\mathbf V$ satisfies $xtx \approx xtx^2$. 
If $\mathbb M_{\gamma}(a^+t) \nsubseteq\mathbf V$, then $\mathbf V$ satisfies  $xtxs \approx xtxsx$.
\end{lemma}

\begin{proof}  
Since  $\bf V$ does not contain $\mathbb M_\gamma (a^+t)$, the $\gamma$-word  $a^+t$  is not a $\gamma$-term for $\mathbf V$ by Lemma~\ref{L: S_alpha in V1}.
Since $t$ is an isoterm for $\mathbf V$, the variety $\mathbf V$ satisfies  $x^n t \approx x^mtx^k$  for some $n, m\ge 2$ and $k \ge 1$. 
In view of $xtx\approx xtx^2$, the variety $\mathbf V$ satisfies  $x^2 t \approx x^2tx$. 
It is easy to check that  the identity $xtxs \approx xtxsx$ is equivalent to $\{xtx \approx xtx^2, x^2t \approx x^2tx\}$.
\end{proof}

The next lemma is a reformulation of Lemma~4.1 in~\cite{Gusev-JAA} using Theorem~7.2  in~\cite{Sapir-21}.

\begin{lemma}
\label{L: does not contain J} 
Let $\mathbf V$ be a monoid variety such that $\mathbb M_{\lambda}(ata^+,a^+t) \subseteq \mathbf V$ and $\mathbf V \models xtx \approx xtx^2$. 
Then $\mathbf V$ satisfies
\begin{equation}
\label{xtyxsy=xtxyxsy}
xtyxsy \approx xtxyxsy
\end{equation}
whenever ${\bf J} = \mathbb M_\lambda(atba^{+}sb^{+}) \nsubseteq \mathbf V$.\qed
\end{lemma}

\begin{lemma} 
\label{L: does not contain K} 
Let $\mathbf V$ be a monoid variety that satisfies the identity $xtx \approx xtx^2$.
If $\mathbf V$ contains $\mathbb M_{\lambda}(ata^+)$ but does not contain ${\bf K} = \mathbb M_\lambda(bta^{+}b^{+})$, then $\mathbf V$ satisfies the identity
\begin{equation} \label{xtyyx=xtyxyx}
xty^2x \approx xty^2xyx.
\end{equation}
%~\eqref{xtysyx=xtysxyx}.
\end{lemma}

\begin{proof}
Since $\mathbb M_\lambda(bta^{+}b^{+})  \nsubseteq\mathbf V$, Lemma~\ref{L: S_alpha in V2} implies that  $bta^{+}b^{+}$ is not a $\lambda$-term for $\mathbf V$. 
Since $ata^+$ is a $\lambda$-term for $\mathbf V$, we get that $\mathbf V$ satisfies $xty^px^q\approx xt\mathbf a$ for some $p \ge 2$, $q\ge1$ and some word $\mathbf a\in\{x,y\}^+$ such that $xy$ is a subword of $\mathbf a$. 
Then the identities
\[
xty^2x\stackrel{xtx \approx xtx^2}\approx xty^2y^{p}x^{q}x\approx xty^2\mathbf ax\stackrel{xtx \approx xtx^2}\approx xty^2xyx
\]
hold in $\mathbf V$, and we are done.
\end{proof}

\begin{lemma}
\label{L: does not contain F}
Let $\mathbf V$ be a monoid variety that satisfies $xtx \approx xtx^2$. 
If $\mathbb M_{\lambda}(ata^+) \nsubseteq\mathbf V$, then $\mathbf V$ satisfies $xtx \approx x^2tx$.
\end{lemma}

\begin{proof} 
If $xy$ is not an isoterm for $\mathbf V$, then $\mathbf V$ is commutative or idempotent by Lemma~\ref{L: xy is not an isoterm}. 
In either case, $\mathbf V$ satisfies $xtx \approx x^2tx$ because $xtx\stackrel{xtx \approx xtx^2}\approx xtx^2\stackrel{\text{comm.}}\approx x^2tx$ and $xtx\stackrel{\text{idemp.}}\approx x^2tx$. So, we may assume that $xy$ is an isoterm for $\mathbf V$.

Since $\mathbb M_{\lambda}(ata^+) \nsubseteq\mathbf V$, Lemma~\ref{L: S_alpha in V2} implies that $ata^+$ is not a $\lambda$-term for $\mathbf V$. 
Then in view of Lemma~\ref{L: xy isoterm}, the variety $\mathbf V$ satisfies an identity of the form $xtx^k \approx x^ptx^q$, where $k,q\ge1$ and $p\ge2$. 
This identity is equivalent modulo $xtx \approx xtx^2$ to $xtx \approx x^2tx$, and we are done.
\end{proof}

\begin{lemma} 
\label{L: does not contain A}  
Let $\mathbf V$ be a monoid variety that satisfies the identities $xtx\approx xtx^2\approx x^2tx$. 
If $\mathbf A^1 \nsubseteq\mathbf V$, then $\mathbf V \models xy^2tx \approx (xy)^2tx$.
\end{lemma}

\begin{proof}  
If $x$ is not an isoterm for $\mathbf V$, then $\mathbf V$ satisfies $x \approx x^2$ and consequently, $xy^2tx \approx (xy)^2tx$.
So, let us assume that $x$ is an isoterm for $\mathbf V$ and consider two cases.

\smallskip

{\bf Case 1}:  $\bf V$ does not contain $\mathbb M_\gamma (a^+t, ta^+)$. 
According to Lemma~\ref{L: S_alpha in V1}, either $a^+t$ or $ta^+$ is not a $\gamma$-term for $\mathbf V$. 
Then the variety $\mathbf V$ satisfies either $x^n t \approx x^mtx^k$ or  $tx^n  \approx x^mtx^k$ for some $n\ge 2$ and $m,k \ge 1$. 
In view of $xtx\approx xtx^2\approx x^2tx$, the variety $\mathbf V$ satisfies either $x^2 t \approx xtx$ or $tx^2  \approx xtx$. 
Each of these identities implies $xy^2tx \approx (xy)^2tx$.

\smallskip

{\bf Case 2}:  $\bf V$ contains $\mathbb M_\gamma (a^+t, ta^+)$.
According to Theorem~4.3(iii) in~\cite{Sapir-21}, we have $\mathbf A^1 =   \mathbb M_{\gamma}(a^+b^+ta^+)$.
Then Lemma~\ref{L: S_alpha in V1} implies that  $a^+ b^+ta^+$ is not a $\gamma$-term for $\mathbf V$.
Since $a^+t$ and $ta^+$ are $\gamma$-terms for $\mathbf V$, the variety $\mathbf V$ satisfies $xy^2tx\approx \mathbf atx$ for some ${\bf a} \in \{x,y\}^+$ such that $\bf a$ contains $yx$ as a subword. 
Then
\[
\mathbf V\models xy^2tx\stackrel{xtx \approx x^2tx}\approx x^2y^3tx\approx x\mathbf aytx\stackrel{xtx \approx xtx^2 \approx x^2tx}\approx (xy)^2tx,
\]
and we are done.
\end{proof}

\begin{sortlemma}
\label{SL: 2}
Let $\mathbf V$ be a variety of aperiodic monoids. 
Then either $\mathbf V$ is HFB or one of the following holds:
\begin{itemize}
\item[\textup{(i)}] $\mathbf V$ contains  one of the varieties $\mathbf A^1\vee\overleftarrow{\mathbf A^1}$, $\mathbf J$, $\overleftarrow{\mathbf J}$, $\mathbf K$, $\overleftarrow{\mathbf K}$, $\mathbf L$ or $\mathbf M$;
\item[\textup{(ii)}] $\mathbf V$ satisfies either  $\{xtx \approx xtx^2, xy^2tx \approx (xy)^2tx\}$  or dually, \\ $\{xtx \approx x^2tx, xty^2x  \approx xt(yx)^2\}$;
\item[\textup{(iii)}] $\mathbf V$ satisfies either  $xtxs \approx xtxsx$  or dually, $txsx \approx xtxsx$.\qed
\end{itemize}
\end{sortlemma}

\begin{proof} 
Suppose that $\mathbf V$ is not HFB and does not contain any of the varieties $\mathbf A^1\vee\overleftarrow{\mathbf A^1}$, $\mathbf J$, $\overleftarrow{\mathbf J}$, $\mathbf K$, $\overleftarrow{\mathbf K}$, $\mathbf L$ or $\mathbf M$. 
Sorting Lemma~\ref{SL: 1} implies that $\mathbf V$ satisfies $xtx \approx xtx^2$ or $xtx \approx x^2tx$. 
By symmetry, we may assume that $\mathbf V$ satisfies $xtx \approx xtx^2$. 
If $x$ is not an isoterm for $\mathbf V$, then $\mathbf V$ is idempotent and consequently, satisfies $xy^2tx \approx (xy)^2tx$. So, we may assume that $x$ is an isoterm for $\mathbf V$.  
If $a^+t$ is not a $\gamma$-term for $\mathbf V$, then $\mathbf V$ satisfies $xtxs \approx xtxsx$ by Lemma~\ref{L: does not contain E}. 
So, we may assume that $a^+t$ is a $\gamma$-term for $\mathbf V$.
Consider two cases.

\smallskip

{\bf Case 1}: $\bf V$ contains $\mathbb M_{\lambda}(ata^+)$. 
Since $\bf V$ does not contain ${\bf J}$,  Lemma~\ref{L: does not contain J} implies that $\mathbf V \models~\eqref{xtyxsy=xtxyxsy}$.  
Erasing $s$ from~\eqref{xtyxsy=xtxyxsy} we obtain 
\begin{equation} 
\label{xtyxy=xtxyxy}
xtyxy \approx xt(xy)^2.
\end{equation}
Since $\mathbf K\nsubseteq\mathbf V$, the variety $\mathbf V$ satisfies~\eqref{xtyyx=xtyxyx} by Lemma~\ref{L: does not contain K}. 
Consequently, $\mathbf V$ satisfies
\[
xsytyx\stackrel{xtx \approx xtx^2}\approx xsyty^2x \stackrel{\eqref{xtyyx=xtyxyx}} \approx xsyty^2xyx\stackrel{xtx \approx xtx^2}\approx xsyt(yx)^2 \stackrel{\eqref{xtyxy=xtxyxy}}\approx xsytxyx
\]
Hence $\mathbf V$ is HFB by Proposition~\ref{P: hfb}.

\smallskip

{\bf Case 2}:  $\bf V$ does not contain $\mathbb M_{\lambda}(ata^+)$. 
In this case, Lemma~\ref{L: does not contain F} implies that $\mathbf V \models xtx \approx x^2tx$.  
Since $\mathbf A^1\vee\overleftarrow{\mathbf A^1}\nsubseteq\mathbf V$,  Lemma~\ref{L: does not contain A} and its dual imply that $\mathbf V$ satisfies one of the following identities
\[
xty^2x \approx xt(yx)^2\ \text{ or }\ xy^2tx \approx (xy)^2tx,
\]
we are done.
\end{proof}

\begin{corollary} 
\label{C: main} 
Let $\mathbf V$ be a variety of $\mathscr J$-trivial monoids. 
Then either $\mathbf V$ is HFB or $\mathbf V$ contains one of the varieties $\mathbf A^1\vee\overleftarrow{\mathbf A^1}$, $\mathbf J$, $\overleftarrow{\mathbf J}$, $\mathbf K$, $\overleftarrow{\mathbf K}$, $\mathbf L$ or $\mathbf M$.
\end{corollary}

\begin{proof} 
Suppose that $\mathbf V$ is not HFB and does not contain any of the varieties $\mathbf A^1\vee\overleftarrow{\mathbf A^1}$, $\mathbf J$, $\overleftarrow{\mathbf J}$, $\mathbf K$, $\overleftarrow{\mathbf K}$, $\mathbf L$ or $\mathbf M$. 
In view of Sorting Lemma~\ref{SL: 2}, two cases are possible.

\smallskip

{\bf Case 1}: $\mathbf V$ satisfies $\{xtx \approx xtx^2, xy^2tx \approx (xy)^2tx\}$ or  $\{xtx \approx x^2tx, xty^2x  \approx xt(yx)^2\}$.
By symmetry, we may assume that the first of these identity system holds in $\mathbf V$.
Since $\mathbf V$ is $\mathscr J$-trivial, $\mathbf V$ satisfies $(xy)^2\approx (yx)^2$  by  Fact~\ref{F: J-trivial}.
Hence, $\mathbf V$ satisfies
\[
xy^2tx \approx (xy)^2tx \stackrel{(xy)^2 \approx (yx)^2}\approx (yx)^2tx\stackrel{xtx \approx xtx^2}\approx (yx)^2xtx \stackrel{(xy)^2 \approx (yx)^2}\approx (xy)^2xtx\approx xy^2xtx.
\]
So, $\mathbf V$ satisfies $xy^2tx \approx xy^2xtx$. Consequently, $\mathbf V$ is HFB by Corollary~\ref{C: hfb1}.

\smallskip

{\bf Case 2}: $\mathbf V$ satisfies either $xtxs \approx xtxsx$  or  $txsx \approx xtxsx$.
Since $\mathbf V$ is $\mathscr J$-trivial, $\mathbf V$ satisfies $(xy)^2\approx (yx)^2$  by  Fact~\ref{F: J-trivial}.
By symmetry, we may assume that the first of these identities holds in $\mathbf V$.
It is routine to verify that 
\[\var\{xtxs \approx xtxsx, (xy)^2\approx (yx)^2\} = \var\{xtx \approx xtx^2, x^2t \approx x^2tx, x^2y^2 \approx y^2x^2\}.
\]
This variety is HFB by Proposition~6.1 in~\cite{Gusev-Vernikov}.
\end{proof}

\section{The last two examples of limit varieties of\\ $\mathscr J$-trivial monoids}
\label{new example}

A monoid $M$ is said to be \textit{non-finitely based} if $\var M$ is NFB. 
As in Section~\ref{suf cond}, we use $_{i{\bf u}}x$ to refer to the $i$th from the left occurrence of $x$ in a word ${\bf u}$. 
We use $_{\ell {\bf u}}x$ to refer to the last occurrence of $x$ in ${\bf u}$.  
If $x$ is simple in $\bf u$ then we use $_{{\bf u}}x$ to denote the only occurrence of $x$ in $\bf u$.
If  the $i$th occurrence of $x$ precedes  the $j$th occurrence of $y$ in a word $\bf u$, we write $({_{i{\bf u}}x}) <_{\bf u} ({_{j{\bf u}}y})$. If  $\mathbf u=\xi(\mathbf s)$ for some endomorphism $\xi$ of $\mathfrak A^*$ and $_{i{\bf u}}x$ is an occurrence of a letter $x$ in $\bf u$ then  $\xi^{-1}_{\bf s}({_{i{\bf u}}x})$ denotes an occurrence  ${_{j{\bf u}}z}$ of a letter $z$ in $\bf s$ such that $\xi({_{j{\bf u}}z})$ regarded as a subword of $\bf u$ contains $_{i{\bf u}}x$.

\begin{sufcon}
\label{SC: for K}
Let $M$ be a monoid such that $M$ satisfies the identity
\begin{equation}
\label{long identity}
{\mathbf u}_n=x y_1^2y_2^2\cdots y^2_{n-1} y_n^2x\approx x y_1^2y_2^2\cdots y^2_{n-1}y_nxy_n={\mathbf v}_n
\end{equation}
for any $n\ge1$.  If the $\lambda$-word $bta^+b^+$ is a $\lambda$-term for $M$ then $M$ is NFB.
\end{sufcon}

\begin{proof}
Let $\bf u$ be a word such that $M \models {\mathbf u}_n \approx {\mathbf u}$ and $({_{\ell{\bf u}}y_n}) <_{\bf u} ({_{2{\bf u}}x})$.
Since the last occurrence of $y_n$ succeeds the second occurrence of $x$ in ${\mathbf v}_n$, in view of Fact~2.1 in~\cite{Sapir-15N}, to show that $M$ is NFB, it suffices to establish that if the identity $\mathbf u\approx \mathbf v$ is directly deducible from some identity $\mathbf s\approx\mathbf t$ of $M$  in less than $n-2$ variables, i.e. $\mathbf u=\mathbf a\xi(\mathbf s)\mathbf b$ and $\mathbf v=\mathbf a\xi(\mathbf t)\mathbf b$ for some words $\mathbf a,\mathbf b\in\mathfrak A^\ast$ and some endomorphism $\xi$ of $\mathfrak A^\ast$, then $({_{\ell{\bf v}}y_n}) <_{\bf v} ({_{2{\bf v}}x})$.

We note that if $a,b\notin\con(\mathbf s\mathbf t)$ then the identity $\mathbf s\approx \mathbf t$ is equivalent to $a\mathbf sb\approx a\mathbf tb$. 
Then there exists an endomorphism $\zeta$ of $\mathfrak A^\ast$ such that $\zeta(a)=\mathbf a$, $\zeta(b)=\mathbf b$, $\zeta(\mathbf s)=\xi(\mathbf s)$ and $\zeta(\mathbf t)=\xi(\mathbf t)$.  
It follows that we may assume without any loss that $\mathbf a=\mathbf b=1$ and so $\mathbf u=\xi(\mathbf s)$ and $\mathbf v=\xi(\mathbf t)$, and $\mathbf s\approx\mathbf t$ is an identity of $M$ in less than $n$ variables.

Since $a^+b^+$ is a $\lambda$-term for $M$ by Fact~\ref{F:2-limited}, we have:
\begin{align} 
\label{letters in u}
&({_{1{\bf u}}x})   <_{\bf u} ({_{\ell{\bf u}}y_1})  <_{\bf u} ({_{1{\bf u}}y_2}) <_{\bf u} ({_{\ell{\bf u}}y_2}) <_{\bf u} \dots <_{\bf u} ({_{1{\bf u}}y_n}) <_{\bf u} ({_{\ell{\bf u}}y_n})  <_{\bf u} ({_{2{\bf u}}x}).
\end{align}
Clearly, the word $ab$ is an isoterm for $M$. Then $\simp(\mathbf s)=\simp(\mathbf t)$, $\mul(\mathbf s)=\mul(\mathbf t)$, $\simp(\mathbf u)=\simp(\mathbf v)$ and $\mul(\mathbf u)=\mul(\mathbf v)$ by Lemma~\ref{L: xy isoterm}. 
We use these facts below without any reference.

Working toward a contradiction, suppose that $({_{2{\bf v}}x}) <_{\bf v} ({_{\ell{\bf v}}y_n})$.
Then $(\xi^{-1}_{\bf t}({_{2{\bf v}}x})) \le_{\bf t} (\xi^{-1}_{\bf t}({_{\ell{\bf v}}y_n}))$, where  $\xi^{-1}_{\bf t}({_{2{\bf v}}x})$ is either the first or the second occurrence of some letter $z$ in $\bf t$ and $\xi^{-1}_{\bf t}({_{\ell{\bf v}}y_n})$ is the last occurrence of some letter $y$ in $\bf v$.
If $y=z$, then $y=z\in\simp(\mathbf s)$ because $\xi(y)=\xi(z)$ contains both $x$ and $y_n$ and  would otherwise appear at least twice in $\bf u$ as a subword, contradicting~\eqref{letters in u}. 
Then, in view of~\eqref{letters in u}, we have $_{\bf s}y={_{\bf s}z}=\xi^{-1}_{\bf s}({_{1{\bf u}}x})$ and ${_{1{\bf t}}z'}=\xi^{-1}_{\bf t}({_{1{\bf v}}x})$ is not an occurrence of $y=z$ in $\mathbf t$. 
Clearly, $({_{1{\bf t}}z'} ) <_{\bf t} ({_{{\bf t}}z})$ but $({_{{\bf s}}z} ) <_{\bf s} ({_{1{\bf s}}z'})$. 
This is impossible, because all the $\lambda$-words in $\{b^+a^+\}^{\le_\lambda}$ are $\lambda$-terms for $M$ by Fact~\ref{F:2-limited}.
So, we may assume that $y \ne z$.
{\sloppy

}
Since all the $\lambda$-words in $\{b^+a^+\}^{\le_\lambda}$ are $\lambda$-terms for $M$, $\mathbf s(z,y)\ne y^iz^j$ for any $i,j\ge0$. Hence $({_{1{\bf s}}z} ) <_{\bf s} ( _{\ell \bf s}y)$.  Using  \eqref{letters in u} and that  $\xi(y)$ contains $y_n$, we obtain: 
 \[({_{1{\bf s}}z} ) <_{\bf s} ( _{\ell \bf s}y) \le_{\bf s}  (\xi^{-1}_{\bf s}({_{\ell{\bf u}}y_n}))  \le_{\bf s}  (\xi^{-1}_{\bf s}({_{2{\bf u}}x})).\]
Since  $\xi(z)$ contains $x$ and  $(\xi^{-1}_{\bf s}({_{1{\bf u}}x}))  \le_{\bf s} ( \xi^{-1}_{\bf s}({_{\ell{\bf u}}y_n}))$ 
by~\eqref{letters in u}, we conclude that  
\begin{equation} 
\label{z in s}
\xi^{-1}_{\bf s}({_{1{\bf u}}x}) = {_{1{\bf s}}z},~ x \in \simp(\xi(z)).
\end{equation}
In view of~\eqref{letters in u} and  the fact that $\bf s$ involves less than $n$ letters, the word $\bf s$ has some letter $c$ such that $\xi(c)$ contains $y_i y_{i+1}$ as a subword for some $1\le i <n$. 
The letter $c$ is simple in $\bf s$ because for each  $1\le i <n$ the word $y_i y_{i+1}$ appears only once in $\bf u$. Using~\eqref{letters in u},~\eqref{z in s}  and that  $\xi(y)$ contains $y_n$, we obtain: 
\begin{equation} 
\label{letters in s}
({_{1{\bf s}}z} ) \le_{\bf s} ( _{\bf s}c)  \le_{\bf s}  (\xi^{-1}_{\bf s}({_{1{\bf u}}y_n})) \le_{\bf s} ({_{1{\bf s}}y} ) \le_{\bf s} ( _{\ell \bf s}y) \le_{\bf s}  (\xi^{-1}_{\bf s}({_{\ell{\bf u}}y_n}))  \le_{\bf s} ( \xi^{-1}_{\bf s}({_{2{\bf u}}x})).
\end{equation}
Two cases are possible.

\smallskip

{\bf Case 1}:  $\xi^{-1}_{\bf t}({_{2{\bf v}}x}) = {_{2{\bf t}}z}$.

In this case, $z$ is multiple in $\bf t$. 
Then $(\xi^{-1}_{\bf s}({_{2{\bf u}}x}))  \le_{\bf s} ({_{2{\bf s}}z} )$, because $\xi(z)$ contains $x$.
Since $c$ is simple in $\bf t$, $z \ne c$.
Using~\eqref{letters in s} and that $z \ne y$ we obtain:
\[
({_{1{\bf s}}z} ) <_{\bf s} ( _{\bf s}c)  \le_{\bf s}   ( {_{1{\bf s}}y})  \le_{\bf s}   ( {_{\ell{\bf s}}y})  <_{\bf s} ({_{2{\bf s}}z}).
\]
Since $bta^+b^+$ is a $\lambda$-term for $M$, it is easily to see that $btb^+$ and $btab^+$ are $\lambda$-terms  for $M$ too. 
Then we have:
\[
({_{1{\bf t}}z} ) <_{\bf t} ( _{\bf t}c) \le_{\bf t}   ( {_{1{\bf t}}y})   \le_{\bf t}   ( {_{\ell{\bf t}}y})  <_{\bf t} ({_{2{\bf t}}z}).
\] 
Our assumption that $({_{2{\bf v}}x}) <_{\bf v} ({_{\ell{\bf v}}y_n})$ implies that 
\[
({_{2{\bf t}}z}) = (\xi^{-1}_{\bf t}({_{2{\bf v}}x}))  \le _{\bf t}   (\xi^{-1}_{\bf t}({_{\ell{\bf v}}y_n})) = ({_{\ell{\bf t}}y}),
\] 
a contradiction.

\smallskip

{\bf Case 2}: $\xi^{-1}_{\bf t}({_{2{\bf v}}x}) = {_{1{\bf t}}z}$.

Then using~\eqref{letters in s} and the fact that all the $\lambda$-words in $\{a^+b^+\}^{\le_\lambda}$ are $\lambda$-terms for $M$, we obtain: 
\begin{equation} 
\label{letters in t}
({_{1{\bf t}}z} ) \le_{\bf t} ( _{\bf t}c) \le_{\bf t} ({_{1{\bf t}}y} ).
\end{equation}

In view of Fact 2.6 in~\cite{Sapir-15N}, $\xi^{-1}_{\bf t}({_{1{\bf v}}x}) = {_{1{\bf t}}d}$ for some  $d \in \con(\bf t) =\con(\bf s)$. 
In this case,  $( _{1\bf t}d)  = (\xi^{-1}_{\bf t}({_{1{\bf v}}x})) \stackrel{\eqref{z in s}} {<_{\bf t}}  (\xi^{-1}_{\bf t}({_{2{\bf v}}x}) )=  ({_{1{\bf t}}z})  \stackrel{\eqref{letters in t}}{\le_{\bf t}} ( _{\bf t}c)$ and therefore $c\ne d$.

Suppose that  $( _{1\bf s}d)  <_{\bf s} ( _{\bf s}c)$. 
Since  $({_{1{\bf s}}z} ) \le _{\bf s} ( _{\bf s}c)  { \le_{\bf s}} ( \xi^{-1}_{\bf s}({_{1{\bf u}}y_n}))$ by~\eqref{letters in s}, and both $\xi(z)$ and $\xi(d)$ contain $x$, we obtain that $({_{2{\bf u}}x}) \le_{\bf u} ({_{\ell{\bf u}}y_n})$, which contradicts to~\eqref{letters in u}. 

Suppose that $( _{\bf s}c)  <_{\bf s} ( _{1\bf s}d)$. 
Then $M$ satisfies ${\bf s}(c,d)= cd^p\approx d^qcd^\ell = {\bf t}(c,d)$ for some $p,q\ge1$ and $\ell\ge0$, which contradicts the fact that $a^+b^+$ is a $\lambda$-term for $M$. 

\smallskip

Therefore, $({_{\ell{\bf v}}y_n}) <_{\bf v} ({_{2{\bf v}}x})$. Then the monoid $M$ is NFB by Fact~2.1 in~\cite{Sapir-15N}.
\end{proof}

\begin{proposition}
\label{P: K is a limit variety}
The varieties $\mathbf K$ and $\overleftarrow{\mathbf K}$ are limit and different from the limit varieties $\mathbf A^1\vee\overleftarrow{\mathbf A^1}$, $\mathbf J$,  $\overleftarrow{\mathbf J}$, $\mathbf L$ and $\mathbf M$.
\end{proposition}

\begin{proof}
It is a routine to verify that $\mathbf K$ satisfies the identity~\eqref{long identity} for any $n\ge1$. 
Then Sufficient Condition and Lemma~\ref{L: S_alpha in V2} imply that $\mathbf K$ is NFB.

It is a routine to check that $\mathbf K$ satisfies the identities $xtx \approx xtx^2$ and $xy^2tx \approx yxytx$. 
Therefore, $\mathbf K$ contains none of the varieties $\overleftarrow{\mathbf J}$, $\overleftarrow{\mathbf K}$, $\mathbf L$ and $\mathbf M$ because they violate $xtx \approx xtx^2$.
Further, $\mathbf K$ does not contain the varieties $\mathbf A^1\vee\overleftarrow{\mathbf A^1}$ and $\mathbf J$ because these varieties do not satisfy $xy^2tx \approx yxytx$.

Since the variety $\mathbf K$ is NFB and does not contain any of the varieties $\mathbf A^1\vee\overleftarrow{\mathbf A^1}$, $\mathbf L$, $\mathbf M$, $\mathbf J$ or $\overleftarrow{\mathbf J}$, Corollary~\ref{C: main} implies that each proper subvariety of $\mathbf K$ is FB. Thus, $\mathbf K$ is a new limit variety. 
\end{proof}

\begin{proof}[Proof of Theorem~\ref{T: main}]
Corollary~\ref{C: main} and  Proposition~\ref{P: K is a limit variety}  imply that a variety of $\mathscr J$-trivial monoids is HFB if and only if it excludes the varieties $\mathbf A^1\vee\overleftarrow{\mathbf A^1}$, $\mathbf J$, $\overleftarrow{\mathbf J}$, $\mathbf K$, $\overleftarrow{\mathbf K}$, $\mathbf L$ and $\mathbf M$. Consequently, there are precisely seven limit varieties of $\mathscr J$-trivial monoids.
\end{proof}

\section{The subvariety lattice of $\mathbf K$}
\label{subvar lattice}

This section is devoted to a description of the subvariety lattice of the limit variety $\mathbf K = \mathbb M_\lambda(bta^{+}b^{+})$. 
The following statement collects some identities of $\mathbf K$ and can be easily verified. We use it sometimes without any reference.

\begin{lemma}
\label{L: identities in K}
The identities $xytxsy\approx yxtxsy$,~\eqref{xzxtxsx=xzxtsx} and~\eqref{xtyxy=xtxyxy} hold in the variety $\mathbf K$.\qed
\end{lemma}

Let $\mathbf A_0^1=\var A_0^1$, where
\[
A_0=\langle e, f \mid e^2=e,\, f^2=f,\,fe=0\rangle=\{e,f,ef,0\}.
\]

\begin{lemma}
\label{L: word problem A_0^1 vee M_lambda(ata^+)}
The variety  $\mathbf A_0^1\vee \mathbb M_\lambda(ata^+)$ satisfies an identity $\mathbf u\approx\mathbf v$ if and only if  each of the following holds:
\begin{itemize}
\item[\textup{(a)}] $\simp ({\bf u}) = \simp ({\bf v})$ and $\mul ({\bf u}) = \mul ({\bf v})$;
\item[\textup{(b)}] for each $x, y \in \con ({\bf u})$ we have $({_{1{\bf u}}x}) <_{\bf u} ({_{\ell{\bf u}}y})$ iff $({_{1{\bf v}}x}) <_{\bf v} ({_{\ell{\bf v}}y})$;
\item[\textup{(c)}] for each $t \in \simp ({\bf u})$ and  $x \in \mul ({\bf u})$
we have 
\[
({_{\bf u}t}) <_{\bf u} ({_{1{\bf u}}x}) \text{ iff } ({_{\bf v}t}) <_{\bf v} ({_{1{\bf v}}x}) \text{ and }({_{\bf u}t}) <_{\bf u} ({_{2{\bf u}}x}) \text{ iff } ({_{\bf v}t}) <_{\bf v} ({_{2{\bf v}}x}).
\]
\end{itemize}
\end{lemma}

\begin{proof}
Put $\mathbf F=\var\{xtx \approx xtx^2,\, x^2t \approx x^2tx,\, x^2y^2 \approx y^2x^2\}$.
According to Theorem~7.1(i) in~\cite{Sapir-21}, we have $\mathbf F= \mathbb M_\lambda(ata^+)$.
Proposition~6.9 in~\cite{Gusev-Vernikov} implies that $\mathbf F$ satisfies an identity
$\mathbf u\approx\mathbf v$ if and only if~(a) and~(c) hold.
Proposition~4.2 in~\cite{Sapir-15} implies that $\mathbf A_0^1$ satisfies an identity
$\mathbf u\approx\mathbf v$ if and only if~(a) and~(b) hold. 
Consequently, $\mathbf A_0^1\vee \mathbb M_\lambda(ata^+)  \models \mathbf u\approx\mathbf v$ if and only if~(a),~(b) and~(c) hold.
\end{proof}

\begin{lemma} 
\label{L: generating var}
\quad
\begin{itemize}
\item[\textup{(i)}] $\mathbf A_0^1 =\mathbb M_{\tau_1}(a^+b^+)= \mathbb M_\lambda(a^+b^+)=\var\{xtsx\approx xtxsx,\, (xy)^2 \approx (yx)^2\}$;
\item[\textup{(ii)}] $\mathbf A_0^1\vee \mathbb M_\lambda(ata^+) = 
\mathbb M_\lambda(ata^+b^+) =$
\[
= \var\{xtx \approx xtx^2,\,xtysxy\approx xtysyx,\, xytxsy \approx yxtxsy,\,  \eqref{xzxtxsx=xzxtsx}\}.
\]
\end{itemize}
\end{lemma} 

\begin{proof}
(i) It follows from \cite[Section~7]{Sapir-18} that the monoid $A_0^1$ is isomorphic to $M_{\tau_1}(a^+b^+)$. 
According to Theorem~4.1(iv) and Fact~6.2 in~\cite{Sapir-21}, the monoid $M_{\lambda}(a^+b^+)$ also generates the variety $\mathbf A_0^1$. 
Proposition~3.2(a) in~\cite{Edmunds-77} says that  $A_0^1$ is finitely based by $\{xtsx\approx xtxsx,\,  xtysxy\approx xtysyx,\, xytxsy \approx yxtxsy\}$. 
It is easy to see that this set of identities is equivalent to $\{xtsx\approx xtxsx,\, (xy)^2 \approx (yx)^2\}$.{\sloppy

}
\smallskip

(ii) Since $\mathbf A_0^1 =\mathbb M_{\lambda}(a^+b^+)$ by Part~(i), we have
\[
\mathbf A_0^1\vee \mathbb M_\lambda(ata^+) \stackrel{\cite[\text{Corollary~2.5}]{Sapir-21}}{=} \mathbb M_\lambda(ata^+, a^+b^+).
\]
The inclusion $\mathbb M_\lambda(ata^+, a^+b^+) \subseteq \mathbb M_\lambda(ata^+b^+)$ follows from Fact~\ref{F:2-limited} and Lemma~\ref{L: S_alpha in V2}. 
The inclusion $\mathbb M_\lambda(ata^+, a^+b^+) \supseteq \mathbb M_\lambda(ata^+b^+)$ holds by the fact that $ata^+b^+$ is a $\lambda$-term for $\mathbb M_\lambda(ata^+, a^+b^+)$ and Lemma~\ref{L: S_alpha in V2}. 
Therefore, $\mathbf A_0^1\vee \mathbb M_\lambda(ata^+)= \mathbb M_\lambda(ata^+b^+)$.

It is routine to check that $\mathbb M_\lambda(ata^+b^+)$ satisfies $xtx\approx xtx^2$ and $xtysxy\approx xtysyx$.
Since $a^+b^+$ is a $\lambda$-term for ${\mathbf K} = \mathbb M_\lambda(bta^{+}b^{+})$ by Fact~\ref{F:2-limited}, $\mathbb M_\lambda(a^+b^+)\subseteq\mathbf K$ by Lemma~\ref{L: S_alpha in V2}. 
Clearly, $ata^+$ is a $\lambda$-term for $\mathbf K$.
It follows that $\mathbb M_\lambda(ata^+)\subseteq{\mathbf K}$ by Lemma~\ref{L: S_alpha in V2}.  
In view of the above, $\mathbb M_\lambda(ata^+b^+)$ is a subvariety of ${\mathbf K}$. Hence the variety $\mathbb M_\lambda(ata^+b^+)$ satisfies  the identities $xytxsy\approx yxtxsy$ and~\eqref{xzxtxsx=xzxtsx} by Lemma~\ref{L: identities in K}.

An identity is called {\em 3-limited} if every letter occurs in each side of it at most~3 times. 
The identity~\eqref{xzxtxsx=xzxtsx} allows us to add and delete the occurrences of a letter $x$ between the second and the last occurrences of $x$, while the identity $xtx \approx xtx^2$ allows us to add and delete the occurrences of  $x$ next to the non-first occurrence of $x$.
Thus, every identity of  $\mathbb M_\lambda(ata^+b^+)$ can be derived from $xtx \approx xtx^2$, \eqref{xzxtxsx=xzxtsx} and a 3-limited identity of $\mathbb M_\lambda(ata^+b^+)$.

Let $\mathbf u\approx\mathbf v$ be a 3-limited  identity of  $\mathbb M_\lambda(ata^+b^+)$.
Since $\mathbf u\approx\mathbf v$ has Properties (a), (b) and (c) in Lemma~\ref{L: word problem A_0^1 vee M_lambda(ata^+)}, it is well-balanced and is a consequence from $\{xtysxy\approx xtysyx,\, xytxsy \approx yxtxsy\}$ by Lemma~4.1 in~\cite{Sapir-15}.
Therefore,
$$
\{xtx \approx xtx^2,\, xytxsy \approx yxtxsy,\,xtysxy\approx xtysyx,\,  \eqref{xzxtxsx=xzxtsx}\}
$$
is an identity basis for $\mathbf A_0^1\vee \mathbb M_\lambda(ata^+) = 
\mathbb M_\lambda(ata^+b^+)$.
\end{proof}

\begin{fact} 
\label{F: does not contain A_0^1}   
Let $\mathbf V$ be a variety of $\mathscr J$-trivial monoids. 
Then $\mathbf V$ does not contain $A_0^1$ if and only if $\mathbf V$ is a variety of aperiodic monoids with commuting idempotents.
\end{fact}

\begin{proof} 
\textit{Necessity}. Suppose that $A_0^1\notin\mathbf V$. Then in view of Lemmas~\ref{L: S_alpha in V1} and~\ref{L: generating var}(i),  the variety $\mathbf V$ satisfies an identity $x^py^k\approx \mathbf w$, where $p, k \ge 1$ and $\mathbf w$ contains $yx$ as a subword. 
Without loss of generality, we may assume that $\con(\mathbf w)=\{x,y\}$.

If $M\in\mathbf V$ and $e,f$ are two idempotents of $M$ then $ef = {\bf w}(e,f)$, where $\mathbf w$ contains $fe$ as a subword. 
If ${\bf w}(e,f) = fe$ then we are done. 
Otherwise, we multiply this equality by $e$ on the left and $f$ on the right and obtain $ef = (ef)^m$ for some $m \ge 2$.

In view of Fact~\ref{F: J-trivial}, $\mathbf V$ satisfies the identities~\eqref{id J-trivial} for some $n\ge1$.  The equality $ef = (ef)^m$ can be iterated to make $m \ge n$. Since $e$ and $f$ are arbitrary idempotents of $M$, we also have $fe = (fe)^m$. 
Therefore,
\[
ef = (ef)^m \stackrel{ \eqref{id J-trivial}}{=}  (fe)^m = fe.
\]

\smallskip

\textit{Sufficiency} follows from the fact that the idempotents $e$ and $f$ of $A_0^1$ do not commute.
\end{proof}

For a monoid variety $\mathbf V$, we denote its subvariety lattice by $L(\mathbf V)$.

\begin{proposition}
\label{P: L(K)}
The lattice $L(\mathbf K)$  has the form shown in Fig.~\ref{pic: L(K)}. 
\end{proposition}

\begin{figure}[htb]
\unitlength=1mm
\linethickness{0.4pt}
\begin{center}
\begin{picture}(55,85)
\put(35,5){\circle*{1.33}}
\put(35,15){\circle*{1.33}}
\put(35,25){\circle*{1.33}}
\put(35,35){\circle*{1.33}}
\put(45,45){\circle*{1.33}}
\put(25,45){\circle*{1.33}}
\put(35,55){\circle*{1.33}}
\put(15,55){\circle*{1.33}}
\put(25,65){\circle*{1.33}}
\put(45,65){\circle*{1.33}}
\put(35,75){\circle*{1.33}}
\put(35,85){\circle*{1.33}}
\put(30,70){\circle*{1.33}}

\put(35,5){\line(0,1){30}}
\put(35,35){\line(-1,1){20}}
\put(35,35){\line(1,1){10}}
\put(25,45){\line(1,1){10}}
\put(25,45){\line(1,1){20}}
\put(15,55){\line(1,1){20}}
\put(45,45){\line(-1,1){20}}
\put(45,65){\line(-1,1){10}}
\put(35,75){\line(0,1){10}}

\put(35,2){\makebox(0,0)[cc]{${\mathbb M}(\varnothing)$}}
\put(37,15){\makebox(0,0)[lc]{${\mathbb M}(1)$}}
\put(37,25){\makebox(0,0)[lc]{${\mathbb M}(x)$}}
\put(37,35){\makebox(0,0)[lc]{${\mathbb M}(xy)$}}
\put(23,45){\makebox(0,0)[rc]{${\mathbb M}_\gamma(ta^+)$}}
\put(14,55){\makebox(0,0)[rc]{${\mathbb M}_\lambda(ata^+)$}}
\put(57,55){\makebox(0,0)[rc]{${\mathbb M}_\gamma(a^+ta^+)$}}
\put(23,65){\makebox(0,0)[rc]{${\mathbb M}_\lambda(a^+ta^+)$}}
\put(61,75){\makebox(0,0)[rc]{${\mathbb M}_\lambda(ata^+b^+)$}}
\put(29,71){\makebox(0,0)[rc]{${\mathbf N} ={\mathbb M}_\lambda(a^+btb^+)$}}
\put(47,45){\makebox(0,0)[lc]{${\mathbb M_\gamma(a^+t)}$}}
\put(47,65){\makebox(0,0)[lc]{$ \mathbf A_0^1 = {\mathbb M}_{\tau_1}(a^+b^+)$}}
\put(37,85){\makebox(0,0)[lc]{${\mathbf K} = {\mathbb M}_\lambda(bta^+b^+)$}}
\end{picture}
\end{center}
\caption{the lattice $L(\mathbf K)$}
\label{pic: L(K)}
\end{figure}
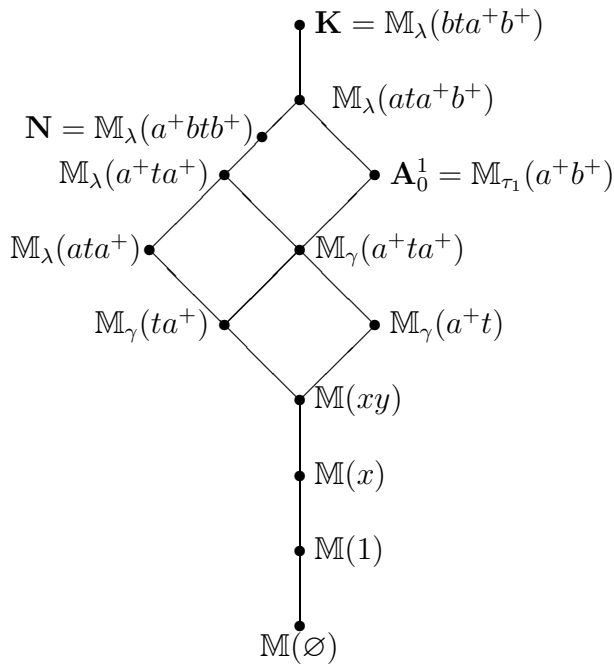

\begin{proof}
Let $\mathbf X$ be a proper subvariety of $\mathbf K$. If $\mathbb M_\lambda(ata^+) \nsubseteq \mathbf X$, then $\mathbf X \models xtx \approx x^2tx$ by Lemma~\ref{L: does not contain F}.  
Then $\mathbf X$ satisfies
\[
xtsx\stackrel{xtx \approx x^2tx}\approx x^2tsx\stackrel{\eqref{xzxtxsx=xzxtsx}}\approx x^2txsx\stackrel{xtx \approx x^2tx}\approx xtxsx.
\]

Since $\{xtsx \approx xtxsx,\, (xy)^2 \approx (yx)^2\}$ is a  basis of identities for $\mathbf A_0^1$  by Lemma~\ref{L: generating var}(i), we have $\mathbf X\subseteq \mathbf A_0^1$. 
The lattice $L(\mathbf A_0^1)$ is as shown in Fig.~\ref{pic: L(K)} by~\cite[Fig.~2]{Lee-12b}.
According to Theorem~4.1 in~\cite{Sapir-21}, the monoids $M_\gamma(a^+ta^+)$, $M_\gamma(ta^+)$ and $M_\gamma(a^+t)$, respectively, generate the same varieties as the monoids $B_0^1$, $I^1$ and $J^1$ in \cite[Fig.~2]{Lee-12b}.
So, we may assume that $\mathbb M_\lambda(ata^+) \subseteq\mathbf X$.{\sloppy

}
If  $\mathbf A_0^1\nsubseteq\mathbf X$, then $\mathbf X$ is a variety of aperiodic monoids with commuting idempotents by Fact~\ref{F: does not contain A_0^1}. 
Since, for any $M\in\mathbf X$ and $x,y\in M$, the elements $x^2$ and $y^2$ are idempotents, we have $\mathbf X\models x^2y^2\approx y^2x^2$. 
Hence, in view of Lemma~\ref{L: identities in K}, $\mathbf X$ is a subvariety of 
\[
\mathbf N= \var\{xtx \approx xtx^2,\, xytxsy \approx yxtxsy, \, x^2y^2\approx y^2x^2, \,  \eqref{xzxtxsx=xzxtsx}\}.
\] 
Put $\mathbf E=\var\{x^2\approx x^3,\,x^2y\approx xyx,\,x^2y^2\approx y^2x^2\}$.
In view of~\cite[Lemma~3.2]{Gusev-SF}, the lattice $L(\mathbf N)$ is the set-theoretical union of the lattice $L({\bf F} \vee \overleftarrow{\bf E})$ and the interval $[{\bf F} \vee \overleftarrow{\bf E}, \mathbf N]$. 
The lattice $L({\bf F} \vee \overleftarrow{\bf E})$  is as shown in Fig.~\ref{pic: L(K)} by~\cite[Fig.~1]{Gusev-SF}.
According to Theorem~7.2(ii) in~\cite{Sapir-21}, the monoid $M_\lambda(a^+ta^+)$ generates the variety ${\bf F} \vee \overleftarrow{\bf E}$.
Thus, we may assume that $\mathbf X\in [\mathbb M_\lambda(a^+ta^+),\mathbf N]$.

Clearly, $\mathbf N$ satisfies~\eqref{xtyxsy=xtxyxsy}. Then Lemmas~3.4 and~3.5 in~\cite{Gusev-JAA} imply that each variety in $[\mathbb M_\lambda(a^+ta^+), \bf N]$ is defined within $\mathbf N$ by some (possibly empty) set of the following identities:~\eqref{xzxtxsx=xzxtsx},
\begin{align*}
xytxy&\approx yxtxy,\\
yx^2txy&\approx xyxtxy,\\
x^2ytxy&\approx xyxtxy,\\
xy\prod_{i=1}^{n+1} (t_i \mathbf e_i)&\approx yx\prod_{i=1}^{n+1} (t_i \mathbf e_i),\\
yx^2\prod_{i=2}^{n+1} (t_i \mathbf e_i)&\approx xyx\prod_{i=1}^n (t_i \mathbf e_i),\\
x^2y\prod_{i=1}^{n+1} (t_i \mathbf e_i)&\approx xyx\prod_{i=1}^{n+1} (t_i \mathbf e_i),\\
x^2y\prod_{i=2}^{n+1} (t_i \mathbf e_i)&\approx xyx\prod_{i=2}^{n+1} (t_i \mathbf e_i),
\end{align*}
where $n \ge1$ and
$$
\mathbf e_i=
\begin{cases}
x&\text{if }i\text{ is odd},\\
y&\text{if }i\text{ is even}.
\end{cases}
$$
The identity~\eqref{xzxtxsx=xzxtsx} holds in $\mathbf N$ by the definition. 
The rest of these identities except for $x^2yty\approx xyxty$ follow from the identity $xytxsy \approx yxtxsy$, which holds in $\mathbf N$.

Therefore, each variety in $[\mathbb M_\lambda(a^+ta^+), \bf N]$ is defined within $\mathbf N$ by the trivial identity (then this variety coincides with $\mathbf N$) or the identity $x^2yty\approx xyxty$ (then this variety equals to $\mathbb M_\lambda(a^+ta^+)$). 
Therefore, $\mathbf X\in \{\mathbb M_\lambda(a^+ta^+),\mathbf N\}$. 

We have $\mathbb M_\lambda(a^+ta^+) = \mathbb M_\lambda(ata^+, a^+t)$ by Theorem~7.2(ii) in~\cite{Sapir-21}.
Since $ata^+$ and $a^+t$ are $\lambda$-terms for  $M_\lambda (a^+btb^+)$, the  variety $\mathbb M_\lambda (a^+btb^+)$  contains $\mathbb M_\lambda(a^+ta^+)$ by  Lemma~\ref{L: S_alpha in V2}. 
Since $\mathbb M_\lambda (a^+ta^+) \models x^2yty \approx xyxty$, the $\lambda$-word $a^+btb^+$ is not a $\lambda$-term for  $\mathbb M_\lambda(a^+ta^+)$. 
Hence, the inclusion  $\mathbb M_\lambda (a^+btb^+) \supset \mathbb M_\lambda(a^+ta^+)$ is proper in view of Lemma~\ref{L: S_alpha in V2}. 
It is routine to verify that ${\mathbb M}_\lambda(a^+btb^+)\subseteq\mathbf N$. 
Then, since ${\mathbb M}_\lambda(a^+btb^+)$ properly contains $\mathbb M_\lambda(a^+ta^+)$, we have ${\mathbf N}={\mathbb M}_\lambda(a^+btb^+)$.

Thus, we may assume that $\mathbf A_0^1\vee \mathbb M_\lambda(ata^+)\subseteq\mathbf X$.  
Then $\bf X$ satisfies the identity~\eqref{xtyyx=xtyxyx} by Lemma~\ref{L: does not contain K}.
Hence, the variety $\mathbf X$ satisfies the identities
\[
xtysyx\!\stackrel{xtx \approx xtx^2}\approx\! xtysy^2x\stackrel{\eqref{xtyyx=xtyxyx}}\approx xtys(yx)^2\!\stackrel{(xy)^2 \approx (yx)^2}\approx\! xtys(xy)^2\stackrel{\eqref{xtyyx=xtyxyx}}\approx xtysx^2y\!\stackrel{xtx \approx xtx^2}\approx\! xtysxy
\]
and so the identity $xsytxy \approx xsytyx$. Since $\bf X$ is a subvariety of $\bf K$, the variety $\bf X$ satisfies the identities in $\{xtx \approx xtx^2,\, xytxsy \approx yxtxsy,\,\eqref{xzxtxsx=xzxtsx}\}$.
Therefore, $\mathbf X = \mathbb M_\lambda(ata^+b^+)$ by Lemma~\ref{L: generating var}(ii).
\end{proof}

\begin{remark}
\label{R: monoid in K}
The variety $\mathbf K$ is generated by the submonoid of $M_\lambda(bta^{+}b^{+})$ generated by $\{a^+,b,ta^+\}$, which is isomorphic to the 12-element monoid $S^1$, where
\begin{align*}
S&=\langle a,b,c\mid a^2=a,\,b^2=b^3,\,abc=ac=ba=b^2c=0,\,bcb^2=bcb,\,ca=c\rangle\\
&=\{a,b,c,ab,ab^2,b^2,bc,bcb,cb,cb^2,0\}.
\end{align*}
\end{remark}

\begin{proof}
It is routine to check that a submonoid of $M_\lambda(bta^{+}b^{+})$ generated by the set $\{a^+,b,ta^+\}$ is isomorphic to $S^1$. Evidently, $S^1\in \mathbf K$. 
We note that $S^1$ violates $xtysxy\approx xtysyx$ because $bca\cdot1\cdot ab=bcb\ne0$ but $bca\cdot1\cdot ba=0$ in $S^1$. 
This fact and Lemmas~\ref{L: identities in K} and~\ref{L: generating var}(ii) imply that $S^1\notin \mathbb M_\lambda(ata^+b^+)$. 
According to Proposition~\ref{P: L(K)}, $\mathbf K$ is generated by $S^1$. 
\end{proof}

\section*{Acknowledgments}
We would like to thank the anonymous referee for providing a careful review with lots of valuable suggestions and a great sense of humor.

Institute of Natural Sciences and Mathematics, Ural Federal University, Lenina 51, Ekaterinburg, 620000, Russia

\textit{E-mail address}: \texttt{sergey.gusb@gmail.com}

\medskip

Department of Mathematics, Vanderbilt University, Nashville, TN 37240, USA

\textit{E-mail address}: \texttt{olga.sapir@gmail.com}


\small

\begin{thebibliography}{99}

%\addcontentsline{toc}{section}{{\noindent\bf Bibliography}}

\bibitem{Edmunds-77} 
Edmunds, C. C. (1977). On certain finitely based varieties of semigroups. \textit{Semigroup Forum} 15(1): 21--39. DOI: 10.1007/BF02195732.

\bibitem{Green-51}
Green, J. A. (1951). On the structure of semigroups. \textit{Ann. Math. (2)} 54(1): 163--172. DOI: 10.2307/1969317.

\bibitem{Gusev-SF}
Gusev, S. V. (2020). A new example of a limit variety of monoids. \textit{Semigroup Forum} 101(1): 102--120. DOI: 10.1007/s00233-019-10078-1.

\bibitem{Gusev-JAA} 
Gusev, S. V. (2021). Limit varieties of aperiodic monoids with commuting idempotents. \textit{J. Algebra Appl.} 20(9): 2150160. DOI: 10.1142/S0219498821501607

\bibitem{Gusev-Vernikov}
Gusev, S. V., Vernikov, B. M. (2018). Chain varieties of monoids. \textit{Dissertationes Math.} 534: 1--73. DOI: 10.4064/dm772-2-2018. 

\bibitem{Head-68}
Head, T. J. (1968). The varieties of commutative monoids. \textit{Nieuw Arch. Wiskunde. III Ser.} 16: 203--206.

\bibitem{Jackson-05} 
Jackson, M. G. (2005). Finiteness properties of varieties and the restriction to finite algebras. \textit{Semigroup Forum} 70(2): 159--187. DOI: 10.1007/s00233-004-0161-x.

\bibitem{Jackson-Lee}
Jackson, M. G., Lee, E. W. H. (2018). Monoid varieties with extreme properties. \textit{Trans. Amer. Math. Soc.} 370(7): 4785--4812. DOI: 10.1090/tran/7091.

\bibitem{Jackson-Sapir} 
Jackson, M. G., Sapir, O. B. (2000). Finitely based, finite sets of words. \textit{Internat. J. Algebra Comput.} 10(6): 683--708. DOI: 10.1142/S0218196700000327.

\bibitem{Lee-09}
Lee, E. W. H. (2009). Finitely generated limit varieties of aperiodic monoids with central idempotents. \textit{J. Algebra Appl.} 8(6): 779--796. DOI: 10.1142/S0219498809003631. 

\bibitem{Lee-12}
Lee, E. W. H. (2012). Maximal Specht varieties of monoids. \textit{Mosc. Math. J.} 12(4): 787--802.  DOI: 10.17323/1609-4514-2012-12-4-787-802.

\bibitem{Lee-12b}
Lee, E. W. H. (2012). Varieties generated by $2$-testable monoids. \textit{Studia Sci. Math. Hungar.} 49(3): 366--389. DOI: 10.1556/sscmath.49.2012.3.1211

\bibitem{Lee-14}
Lee, E. W. H. (2014). On certain Cross varieties of aperiodic monoids with commuting idempotents. \textit{Results in Math.} 66(3--4): 491--510. DOI: 10.1007/s00025-014-0390-6.

\bibitem{Lee-Li}
Lee, E.W.H., Li, J. R. (2011). Minimal non-finitely based monoids. \textit{Dissertationes Math.} 475: 1--65. DOI: 10.4064/dm475-0-1. 

\bibitem{Lee-Zhang}  
Lee,  E. W. H., Zhang, W. T. (2015). Finite basis problem for semigroups of order six. \textit{LMS J. of Comp. and Math.} 18(1): 1--129. DOI:10.1112/S1461157014000412.

\bibitem{MH} 
Morse, M., Hedlund, G. A. (1944). Unending chess, symbolic dynamics and a problem in semigroups. \textit{Duke Math. J.} 11(1): 1--7. DOI: 10.1215/S0012-7094-44-01101-4. 

\bibitem{Perkins-69} 
Perkins, P. (1969). Bases for equational theories of semigroups. \textit{J. Algebra} 11(2): 298--314. DOI: 10.1016/0021-8693(69)90058-1.

\bibitem{Sapir-00}
Sapir, O. B. (2000). Finitely based words. \textit{Internat. J. Algebra Comput.} 10(4): 457--480. DOI: 10.1142/S0218196700000224. 
 
\bibitem{Sapir-15} 
Sapir, O. B. (2015). Finitely based monoids. \textit{Semigroup Forum} 90(3): 587--614. DOI: 10.1007/s00233-015-9708-2.

\bibitem{Sapir-15N} 
Sapir, O. B. (2015). Non-finitely based monoids. \textit{Semigroup Forum} 90(3): 557--586. DOI: 10.1007/s00233-015-9709-1.

\bibitem{Sapir-18} 
Sapir, O. B. (2018). Lee monoids are non-finitely based while the sets of their isoterms are finitely based. \textit{Bull. Aust. Math. Soc.} 97(3): 422--434. DOI: 10.1017/S0004972718000023.

\bibitem{Sapir-21} 
Sapir, O. B. (2021). Limit varieties of $J$-trivial monoids. \textit{Semigroup Forum} 103(1): 236--260. DOI: 10.1007/s00233-021-10183-0.

\bibitem{Volkov-90} 
Volkov, M. V. (1990). A general finite basis condition for system of semigroup identities. \textit{Semigroup Forum} 41(1): 181--191. DOI: 10.1007/BF02573389.

\bibitem{Volkov-01}
Volkov, M. V. (2001). The finite basis problem for finite semigroups. \textit{Sci. Math. Jpn.} 53(1): 171--199. 

\bibitem{Wismath-86}
Wismath, S. L. (1986). The lattice of varieties and pseudovarieties of band monoids. \textit{Semigroup Forum} 33(1): 187--198. DOI: 10.1007/BF02573192.

\bibitem{Zhang-13}
Zhang, W. T. (2013). Existence of a new limit variety of aperiodic monoids. \textit{Semigroup Forum} 86(1): 212--220. DOI: 10.1007/s00233-012-9410-6.

\bibitem{Zhang-Luo} 
Zhang, W. T., Luo, Y. F. A new example of limit variety of aperiodic monoids. Manuscript. Available at: https://arxiv.org/abs/1901.02207

\end{thebibliography}
\end{document}